\documentclass[12pt,reqno]{amsart}
\usepackage[letterpaper,margin=1in,footskip=0.25in]{geometry}
\usepackage{amssymb}
\usepackage[only,llbracket,rrbracket]{stmaryrd}
\usepackage{microtype}
\usepackage{mathtools}
\usepackage{enumitem}
\usepackage{booktabs}
\usepackage[noadjust]{cite}
\renewcommand{\citepunct}{;\penalty\citemidpenalty\ }

\PassOptionsToPackage{dvipsnames}{xcolor}
\usepackage{tikz-cd}

\PassOptionsToPackage{pdfusetitle,colorlinks,pagebackref,linktocpage,bookmarksdepth=3}{hyperref}
\usepackage{bookmark}
\hypersetup{citecolor=OliveGreen,linkcolor=Mahogany,urlcolor=Plum}
\usepackage[hyphenbreaks]{breakurl}

\newtheorem{theorem}{Theorem}[section]
\newtheorem{corollary}[theorem]{Corollary}
\newtheorem{lemma}[theorem]{Lemma}
\newtheorem{proposition}[theorem]{Proposition}

\newtheorem{innercustomthm}{Theorem}
\newenvironment{customthm}[1]
  {\renewcommand\theinnercustomthm{#1}\innercustomthm}
  {\endinnercustomthm}

\theoremstyle{definition}
\newtheorem{definition}[theorem]{Definition}
\newtheorem{example}[theorem]{Example}

\theoremstyle{remark}
\newtheorem{remark}[theorem]{Remark}

\newtheoremstyle{cited}{.5\baselineskip\@plus.2\baselineskip\@minus.2\baselineskip}{.5\baselineskip\@plus.2\baselineskip\@minus.2\baselineskip}{\itshape}{}{\bfseries}{\bfseries .}{5pt plus 1pt minus 1pt}{\thmname{#1}\thmnumber{ #2}\thmnote{ \normalfont#3}}
\theoremstyle{cited}
\newtheorem{citedthm}[theorem]{Theorem}
\newtheorem{citedlem}[theorem]{Lemma}

\newtheoremstyle{citeddef}{.5\baselineskip\@plus.2\baselineskip\@minus.2\baselineskip}{.5\baselineskip\@plus.2\baselineskip\@minus.2\baselineskip}{}{}{\bfseries}{\bfseries .}{5pt plus 1pt minus 1pt}{\thmname{#1}\thmnumber{ #2}\thmnote{ \normalfont#3}}
\theoremstyle{citeddef}
\newtheorem{citeddef}[theorem]{Definition}

\DeclareMathOperator{\Ext}{Ext}

\DeclareMathOperator{\Min}{Min}
\DeclareMathOperator{\Spec}{Spec}
\DeclareMathOperator{\height}{ht}
\DeclareMathOperator{\FI}{FI}

\newcommand{\FF}{\mathbf{F}}
\newcommand{\PP}{\mathbf{P}}
\newcommand{\cO}{\mathcal{O}}
\newcommand{\cP}{\mathcal{P}}
\newcommand{\fm}{\mathfrak{m}}
\newcommand{\fn}{\mathfrak{n}}
\newcommand{\fp}{\mathfrak{p}}
\newcommand{\fq}{\mathfrak{q}}
\newcommand{\id}{\mathrm{id}}
\newcommand{\red}{\mathrm{red}}

\begin{document}
\title{Permanence properties of \emph{F}-injectivity}

\author{Rankeya Datta}
\thanks{RD was supported by an AMS-Simons travel grant.}
\address{Department of Mathematics\\University of Missouri\\Columbia, MO 65212\\USA}
\email{\href{mailto:rankeya.datta@missouri.edu}{rankeya.datta@missouri.edu}}
\urladdr{\url{https://www.rankeyadatta.com}}

\author{Takumi Murayama}
\thanks{TM was supported by the National Science
Foundation under Grant Nos.\ DMS-1701622 and DMS-1902616}
\address{Department of Mathematics\\Purdue University\\West Lafayette, IN
47907-2067\\USA}
\email{\href{mailto:murayama@purdue.edu}{murayama@purdue.edu}}
\urladdr{\url{https://www.math.purdue.edu/~murayama/}}

\makeatletter
  \hypersetup{
    pdfauthor={Rankeya Datta and Takumi Murayama},
    pdfsubject=\@subjclass,pdfkeywords={F-injective ring, local cohomology, base change, generic projection}
  }
\makeatother

\begin{abstract}
We prove that $F$-injectivity localizes, descends under faithfully flat homomorphisms, and ascends under flat homomorphisms with Cohen--Macaulay and geometrically $F$-injective fibers, all for arbitrary Noetherian rings of prime characteristic. As a consequence, we show that the $F$-injective locus is open on most rings arising in arithmetic and geometry.
As a geometric application, we prove that over an algebraically closed field of
characteristic $p > 3$, generic projection hypersurfaces associated to suitably
embedded smooth projective varieties of dimension $\le 5$ are $F$-pure, and
hence $F$-injective.
This geometric result is the positive characteristic analogue of a theorem of Doherty.
\end{abstract}

\maketitle

\setcounter{tocdepth}{1}
\tableofcontents

\section{Introduction}
Let $\cP$ be a property of Noetherian local rings, and denote by $(R,\fm)$ a Noetherian local ring. Reasonably well-behaved properties $\cP$ usually satisfy the following:
\begin{enumerate}[label=(\Roman*),ref=\Roman*]
  \item\label{property:localization}
    (Localization) $R$ is $\cP$ if and only if $R_\fp$ is $\cP$ for every
    prime ideal $\fp \subseteq R$.
  \item\label{property:descent}
    (Descent) If $(R,\fm) \to (S,\fn)$ is a flat local homomorphism of
    Noetherian local rings, and $S$ is $\cP$, then $R$ is $\cP$.
  \item\label{property:ascent}
    (Ascent/Base change) If $(R,\fm) \to (S,\fn)$ is a flat local homomorphism of Noetherian local rings, and both $R$ and the closed fiber $S/\fm S$ are $\cP$, then $S$ is $\cP$.
  \item\label{property:openness}
    (Openness) If $(R, \fm)$ is excellent local, then the locus $\{\fq \in \Spec(R) \mid
    R_\fq\ \text{is}\ \cP\}$ is open.
\end{enumerate}
As an illustration, (\ref{property:localization})--(\ref{property:openness}) hold for many classical properties $\cP$:
\begin{center}
  \scriptsize
  \begin{tabular}{cccc}
    \toprule
    $\cP$ & (\ref{property:localization}) & (\ref{property:descent}) $+$
    (\ref{property:ascent}) & (\ref{property:openness})\\
    \cmidrule(lr){1-1} \cmidrule(lr){2-4}
    regular & \cite[Thm.\ 19.3]{Mat89} & \cite[Thm.\ 23.7]{Mat89} & \cite[Def.\
    on p.\ 260]{Mat89}\\
    normal & \cite[Def.\ on p.\ 64]{Mat89} & \cite[Cor.\ to Thm.\ 23.9]{Mat89} &
    \cite[Prop.\ 7.8.6$(iii)$]{EGAIV2}\\
    reduced & \cite[Prop.\ II.2.17]{BouCA} & \cite[Cor.\ to Thm.\ 23.9]{Mat89} &
    \cite[Prop.\ 7.8.6$(iii)$]{EGAIV2}\\
    complete intersection & \cite[Cor.\ 1]{Avr75} & \cite[Thm.\ 2]{Avr75} &
    \cite[Cor.\ 3.3]{GM78}\\
    Gorenstein & \cite[Thm.\ 18.2]{Mat89} & \cite[Thm.\ 23.4]{Mat89} &
    \cite[Rem.\ $(b)$ on p.\ 213]{GM78}\\
    Cohen--Macaulay & \cite[Thm.\ 17.3$(iii)$]{Mat89} & \cite[Cor.\ to Thm.\
    23.3]{Mat89} & \cite[Prop.\ 7.8.6$(iii)$]{EGAIV2}\\
    \bottomrule
  \end{tabular}
 \end{center}

\medskip
\par In this paper, we are interested in assertions
(\ref{property:localization})--(\ref{property:openness})
for $F$-singularities.
The theory of $F$-singularities arose in the work of Hochster--Roberts
\cite{HR76} and Hochster--Huneke \cite{HH90} in tight closure theory, and in the work of Mehta--Ramanathan \cite{MR85} and Ramanan--Ramanathan \cite{RR85} in the theory of Frobenius splittings. Part of their motivation was to detect singularities of rings of prime characteristic $p > 0$ using the Frobenius homomorphism $F_R\colon R
\to F_{R*}R$, a starting point for which was Kunz's characterization of  regularity of a Noetherian ring $R$ in terms of the flatness of $F_R$ \cite[Thm.\ 2.1]{Kun69}.
The most common classes of $F$-singularities are related in the following
fashion (see Remark \ref{rem:implications} for details):
\[
  \begin{tikzcd}
    \text{strongly $F$-regular}
    \rar[Rightarrow] \dar[Rightarrow] &
    \text{$F$-rational} \dar[Rightarrow]\\
    \text{$F$-pure} \rar[Rightarrow]
    & \text{$F$-injective}\mathrlap{.}
  \end{tikzcd}
\]
For three of the four classes of $F$-singularities listed above, we know that assertions (\ref{property:localization})--(\ref{property:openness}) hold in some situations:
\begin{center}\scriptsize
  \begin{tabular}{ccccc}
    \toprule
    $\cP$ & (\ref{property:localization}) & (\ref{property:descent}) &
    (\ref{property:ascentstar}) & (\ref{property:openness})\\
    \cmidrule(lr){1-1} \cmidrule(lr){2-5}
    strongly $F$-regular & \cite[Lem.\ 3.6]{Has10} & \cite[Lem.\
    3.17]{Has10} & \cite[Lem.\ 3.28]{Has10} & \cite[Prop.\ 3.33]{Has10}\\
    $F$-rational & Proposition \ref{prop:frational}$(\ref{prop:frationalcm})$ &
    Proposition \ref{prop:frationaldescends} 
        & \cite[Thm.\ 3.1]{Vel95} & \cite[Thm.\ 3.5]{Vel95}\\
    $F$-pure & \cite[Lem.\ 6.2]{HR74} & \cite[Prop.\ 5.13]{HR76} & \cite[Prop.\
    2.4(4)]{Has10} & \cite[Cor.\ 3.5]{Mur}\\
    \bottomrule
  \end{tabular}
\end{center}
\smallskip
Here, strong $F$-regularity is defined in terms of tight closure as in
\cite[Def.\ 3.3]{Has10} for Noetherian rings that are not necessarily $F$-finite
(following Hochster), and (\ref{property:ascentstar}) is the
following special case of (\ref{property:ascent}):
\begin{enumerate}[label=(\Roman*),ref=\Roman*]
  \item[(\ref*{property:ascent}*)]
    \refstepcounter{enumi}
    \makeatletter
    \def\@currentlabel{\ref*{property:ascent}*}
    \makeatother
    (Ascent/Base change for regular homomorphisms) If $(R,\fm) \to (S,\fn)$ is a flat
    local homomorphism of Noetherian local rings with geometrically regular fibers,
    and $R$ is $\cP$, then $S$ is $\cP$.\label{property:ascentstar}
\end{enumerate}
Property (\ref{property:ascent}) and not just (\ref{property:ascentstar}) also holds for $F$-purity by \cite[Prop.\ 4.8]{SZ13}.
In the $F$-finite setting, where strong $F$-regularity was first defined by
Hochster and Huneke via a splitting condition, (\ref{property:localization}) is
follows from \cite[Thm.\ 5.5$(a)$]{HH94}, (\ref{property:descent}) follows from
\cite[Thm.\ 5.5$(b)$]{HH94}, and (\ref{property:openness}) follows from
\cite[Thm.\ 5.9$(b)$]{HH94}.
Furthermore, (\ref{property:ascent}) holds for $F$-finite strongly $F$-regular
rings by adapting the argument in \cite[Thm.\ 7.3]{HH94};
see \cite[Cor.\ 4.6]{SZ13}. 

We note that (\ref{property:ascentstar}) for strong $F$-regularity only holds under the additional assumption that $S$ is excellent, and (\ref{property:localization}) (resp.\ (\ref{property:ascentstar})) for $F$-rationality only holds under the additional assumption that $R$ is the image of a Cohen--Macaulay ring (resp.\ that $R$ and $S$ are excellent). In all three cases, (\ref{property:openness}) is also known to hold for rings essentially of finite type over excellent local rings, but not for arbitrary excellent rings. 

\par For $F$-injectivity,
(\ref{property:localization}) and
(\ref{property:openness}) appear in the literature in the $F$-finite setting
\citeleft\citen{Sch09}\citemid Prop.\ 4.3\citepunct \citen{QS17}\citemid Prop.\
3.12\citeright, and (\ref{property:descent}) follows by combining
(\ref{property:localization}) in the $F$-finite case with \cite[Lemma A.3]{Mur}.
On the other hand, (\ref{property:ascentstar})
seems to be completely open when the base ring is not assumed to be Cohen--Macaulay. We recall that a ring $R$ of prime characteristic $p > 0$ is \textsl{$F$-injective} if, for every maximal ideal $\fm \subseteq R$, the Frobenius action $H^i_\fm(F_{R_\fm})\colon H^i_\fm(R_\fm) \to H^i_\fm(F_{R_\fm*}R_\fm)$ is injective for every $i$.
$F$-injective rings are related to rings with Du Bois singularities in
characteristic zero \cite{Sch09}.
\medskip

\par The aim of this paper is to address assertions
(\ref{property:localization})--(\ref{property:openness}) for $F$-injectivity. For (\ref{property:localization}), we show that $F$-injectivity localizes for arbitrary Noetherian rings (Proposition \ref{prop:finjlocalizes}), extending results of Schwede \cite[Prop.\ 4.3]{Sch09} and Hashimoto \cite[Cor.\ 4.11]{Has10}. We then prove (\ref{property:descent}), i.e., that $F$-injectivity descends under faithfully flat homomorphisms of Noetherian rings (Theorem \ref{thm:descentF-injective}). This latter statement extends results of Hashimoto \citeleft\citen{Has01}\citemid Lem.\ 5.2\citepunct \citen{Has10}\citemid Lem.\ 4.6\citeright\ and the second author \cite[Lem.\ A.3]{Mur}.

Our main theorem resolves (\ref{property:ascentstar}) for $F$-injectivity by proving a more general statement.

\begin{customthm}{\ref{thm:finjascendscmfi}}
Let $\varphi\colon (R,\fm) \to (S,\fn)$ be a flat local homomorphism of
Noetherian local rings of prime characteristic $p > 0$ whose closed fiber $S/\fm S$ is Cohen--Macaulay and geometrically $F$-injective over $R/\fm$. If $R$ is $F$-injective, then $S$ is $F$-injective.
\end{customthm}

\noindent Theorem \ref{thm:finjascendscmfi}, and its non-local version (Corollary \ref{cor:finjascendscmfinonlocal}), are new even when $\varphi$ is regular (i.e.\ flat with geometrically regular fibers) or smooth.
Under the additional hypothesis that $R$ is Cohen--Macaulay, Theorem
\ref{thm:finjascendscmfi} is due to Hashimoto \cite[Cor.\ 5.7]{Has01} and Aberbach--Enescu \cite[Thm.\ 4.3]{Ene09}. We note that (\ref{property:ascent}) fails when $\cP$ is the property of being $F$-injective, even if the closed fiber $S/\fm S$ is regular \cite[\S4]{Ene09}. This indicates that some geometric assumptions  on the closed fiber $S/\fm S$ are needed for a version of Theorem \ref{thm:finjascendscmfi}, and consequently, (\ref{property:ascent}), to hold for $F$-injectivity.
\par The Cohen--Macaulayness of $S/\fm S$ is used to decompose the local cohomology of
$S$ as a tensor product (see \S\ref{sect:ascent}), and also applies the
characterization of Cohen--Macaulay $F$-injective rings using Frobenius closure
essentially due to Fedder and Watanabe \cite[Rem.\ 1.9 and Prop.\ 2.2]{FW89} to
the closed fiber $S/\fm S$ (see
Lemma \ref{lem:cmfiparam}).
If one could show that Theorem \ref{thm:finjascendscmfi} holds
without this assumption, then this would also show that $F$-injectivity deforms (see the
proof of \cite[Lem.\ 5.9]{Has01}). The latter remains one of the most important open problems about $F$-injectivity, and Theorem \ref{thm:finjascendscmfi} provides evidence in support of this problem since we are now able to drop the Cohen--Macaulay assumption on the base ring. \medskip

\par As an application of our main theorem, we prove that $F$-injectivity is an open condition for most rings that arise in geometric applications, in particular resolving (\ref{property:openness}) and answering a question of the second author \cite[Rem.\ 3.6]{Mur}.

\begin{customthm}{\ref{thm:finjlocusopen}}
Let $R$ be a ring essentially of finite type over a Noetherian local ring $(A,\fm)$ of prime characteristic $p > 0$, and suppose that $A$ has Cohen--Macaulay and geometrically $F$-injective formal fibers. Then, the $F$-injective locus is open in $\Spec(R)$.
\end{customthm}

\noindent The condition on formal fibers holds if $A$ is excellent. At the same time, we also observe that the $F$-injective locus need not be open if $R$ is merely locally excellent (see Example \ref{ex:finjnotopen}).

\medskip
\par Finally, we give a geometric application of our results.
Bombieri \cite[p.\
209]{Bom73} and Andreotti--Holm \cite[p.\ 91]{AH77} asked whether the image of a
smooth projective variety of dimension $r$ under a generic projection to
$\PP^{r+1}_k$ is weakly normal.
Greco--Traverso \cite[Thm.\ 3.7]{GT80} proved that this is indeed the case over
the complex numbers, and the case in arbitrary characteristic follows from work
of Roberts--Zaare-Nahandi
\citeleft\citen{ZN84}\citemid Thm.\ 3.2\citepunct\citen{RZN84}\citemid Thm.\
1.1\citeright\ and Cumino--Manaresi \cite[Thm.\ 3.8]{CM81} (see \cite[Obs.\
2.4$(v)$]{CGM89}).
As an application of our results, we show that generic projection hypersurfaces
associated to suitably embedded smooth varieties of low dimension are in fact
$F$-pure (hence also $F$-injective) in positive characteristic, which is a
stronger condition than being weakly normal \cite[Thm.\ 7.3]{SZ13}.

\begin{customthm}{\ref{thm:dohertyanalogue}}
Let $Y \subseteq \PP^n_k$ be a smooth projective variety of dimension $r \le 5$ over an algebraically closed field $k$ of characteristic $p > r$, such that $Y$ is embedded via the $d$-uple embedding with $d \ge 3r$. If $\pi\colon Y \to \PP^{r+1}_k$ is a generic projection and $X = \pi(Y)$, then $X$ is $F$-pure, and hence $F$-injective.
\end{customthm}

\noindent The assumption on characteristic is needed to rule out some
exceptional cases in the classification in \cite[(13.2)]{Rob75}, and the $d$-uple embedding is used to ensure that $Y$ is appropriately embedded in the sense of \cite[\S9]{Rob75}. Theorem \ref{thm:dohertyanalogue} is the positive characteristic analogue of a theorem of Doherty \cite[Main Thm.]{Doh08}, who proved that over the complex numbers, the generic projections $X$ have semi-log canonical singularities. 
An
example of Doherty also shows that generic projections of smooth projective varieties
of large dimension are not necessarily $F$-pure (see Example
\ref{ex:dohertyhighdim}).

\subsection*{Notation}
All rings will be commutative with identity.
If $R$ is a ring of prime characteristic $p > 0$, then the \textsl{Frobenius
homomorphism} on $R$ is the ring homomorphism
\[
  \begin{tikzcd}[column sep=1.475em,row sep=0]
    \mathllap{F_R\colon} R \rar & F_{R*}R\\
    r \rar[mapsto] & r^p\mathrlap{.}
  \end{tikzcd}
\]
The notation $F_{R*}R$ is used to emphasize the fact the target of the Frobenius homomorphism has the (left) $R$-algebra structure given by $a \cdot r = a^pr$. 
For every integer $e \ge 0$, we denote the $e$-th iterate of the Frobenius homomorphism by $F^e_R\colon R \to F^e_{R*}R$.
If $I \subseteq R$ is an ideal, we define the \textsl{$e$-th Frobenius power}
$I^{[p^e]}$ to be the ideal $I \cdot F^e_{R*}R \subseteq F^e_{R*}R$
generated by $p^e$-th powers of elements in $I$.

\par Local cohomology modules are defined by taking injective resolutions in the category of sheaves of Abelian groups on spectra, as is done in \cite[Exp.\ I, D\'ef.\ 2.1]{SGA2}. When local cohomology is supported at a finitely generated ideal, this definition matches the \v{C}ech or Koszul definitions for local cohomology, even without Noetherian hypotheses \cite[Exp.\ II, Prop.\ 5]{SGA2}.

\subsection*{Acknowledgments}
We are grateful to Florian Enescu, Mihai Fulger, Mitsuyasu Hashimoto, Jack Jeffries, J\'anos Koll\'ar, Linquan Ma, Lance Edward Miller, Thomas Polstra, Karl Schwede, Kazuma Shimomoto, Kevin Tucker, and Farrah Yhee for helpful discussions. We would especially like to thank Melvin Hochster for insightful conversations and comments on early drafts of this paper. We are grateful to Farrah Yhee for pointing out the reference \cite{Ngu12}. The second author would also like to thank his advisor Mircea Musta\c{t}\u{a} for his constant support and encouragement. We thank the referees for their careful reading of the paper and helpful comments.

\section{Definitions and preliminaries}\label{sect:defsprelim}
We collect some basic material on $F$-injective rings and on the relative Frobenius homomorphism, which will be essential in the rest of this paper.

\subsection{\emph{F}-injective and CMFI rings} 
We start by defining $F$-injective rings, which were introduced by Fedder
\cite{Fed83}.

\begin{citeddef}[{\citeleft\citen{Fed83}\citemid Def.\ on p.\
  473\citepunct\citen{Ene00}\citemid p.\ 546\citeright}]\label{def:FI}
A Noetherian local ring $(R,\fm)$ of prime characteristic $p > 0$ is \textsl{$F$-injective} if for all integers $i \geq 0$, the $R$-module homomorphism
  \[
    H^i_\fm(F_R) \colon H^i_\fm(R) \longrightarrow H^i_\fm(F_{R*}R)
  \]
  is injective.
  An arbitrary Noetherian ring $R$ of prime characteristic $p > 0$ is \textsl{$F$-injective} if $R_\fm$ is $F$-injective for every maximal ideal $\fm \subseteq R$. A locally Noetherian scheme $X$ of prime characteristic $p > 0$ is \textsl{$F$-injective} if there exists an affine open cover $\{\Spec(R_\alpha)\}_\alpha$ of $X$ such that
  $R_\alpha$ is $F$-injective
  for all $\alpha$.
  
\par Now let $R$ be a Noetherian $k$-algebra, where $k$ is a field of characteristic $p > 0$. We say that $R$ is \textsl{geometrically $F$-injective over $k$} if for every finite purely inseparable extension $k \subseteq k'$, the ring $R \otimes_k k'$ is $F$-injective. Similarly, if $X$ is a locally Noetherian $k$-scheme, we say that $X$ is \textsl{geometrically $F$-injective over $k$} if for every finite purely inseparable extension $k \subseteq k'$, the scheme $X \times_k \Spec(k')$ is $F$-injective.
\end{citeddef}

\begin{remark}
Our definition of geometric $F$-injectivity follows \cite[p.\ 546]{Ene00}. The definition in \cite[Def.\ 5.3]{Has01} is a priori stronger, since it asserts that $R \otimes_k k'$ is $F$-injective for \emph{all} finite extensions $k \subseteq k'$. We will see that these definitions are equivalent (Proposition \ref{prop:geomfinjextensions}$(\ref{prop:geomfinjseparable})$).
\end{remark}

We also define the following:

\begin{citeddef}[{\citeleft\citen{Has01}\citemid p.\ 238\citepunct\citen{Has10}\citemid (4.1)\citeright}]\label{def:cmfi}
A Noetherian ring $R$ of prime characteristic $p > 0$ is \textsl{Cohen--Macaulay $F$-injective} or \textsl{CMFI} if $R$ is both Cohen--Macaulay and $F$-injective. Similarly, a locally Noetherian scheme $X$ of prime characteristic $p > 0$ is \textsl{CMFI} if $X$ is Cohen--Macaulay and $F$-injective.

\par A Noetherian $k$-algebra (resp.\ a locally Noetherian $k$-scheme) is
\textsl{geometrically CMFI over $k$} if it is Cohen--Macaulay and geometrically
$F$-injective over $k$.
\end{citeddef}

\begin{remark}\label{rem:gcmiscm}
Since the notions of ``geometrically Cohen--Macaulay'' and ``Cohen--Macaulay'' coincide \cite[Rem.\ on p.\ 182]{Mat89}, a Cohen--Macaulay $k$-algebra $R$
is geometrically CMFI as in Definition \ref{def:cmfi} if and only if for every finite purely inseparable extension $k \subseteq k'$, the ring $R \otimes_k k'$
is $F$-injective.
\end{remark}

We will use a characterization of CMFI rings in terms of Frobenius closure of ideals.

\begin{citeddef}[{\cite[(10.2)]{HH94}}]
Let $R$ be a ring of prime characteristic $p > 0$. If $I \subseteq R$ is an ideal, then the \textsl{Frobenius closure} of $I$ in $R$ is
\[
I^F \coloneqq \bigl\{ x \in R \bigm\vert x^{p^e} \in I^{[p^e]}\ \text{for some}\ e > 0 \bigr\}.
\]
We say that $I$ is \textsl{Frobenius closed} if $I = I^F$.
\end{citeddef}

CMFI rings can be characterized in terms of Frobenius closure of ideals generated by systems of parameters.

\begin{citedlem}[{\citeleft\citen{FW89}\citemid Rem.\ 1.9 and Prop.\ 2.2\citepunct\citen{Has10}\citemid Lem.\ 4.4\citepunct\citen{QS17}\citemid Thm.\ 3.7 and Cor.\ 3.9\citeright}]\label{lem:cmfiparam}
Let $(R,\fm)$ be a Cohen--Macaulay local ring of prime characteristic $p > 0$. Then, the following are equivalent:
  \begin{enumerate}[label=$(\roman*)$,ref=\roman*]
    \item The ring $R$ is $F$-injective.\label{lem:cmfiparamfinj}
    \item Every ideal generated by a system of parameters for $R$ is Frobenius closed.\label{lem:cmfiallparam}
    \item There exists an ideal generated by a system of parameters for $R$ that is Frobenius closed.\label{lem:cmfioneparam}
  \end{enumerate}
Moreover, even if $R$ is not Cohen--Macaulay, we have $(\ref{lem:cmfiallparam}) \Rightarrow (\ref{lem:cmfiparamfinj})$.
\end{citedlem}

\subsection{The relative Frobenius homomorphism}
We recall the definition of the \textsl{relative Frobenius homomorphism} \cite[Exp.\ XV, D\'ef.\ 3]{SGA5}, which is also known as the \textsl{Radu--Andr\'e homomorphism} in the commutative algebraic literature.

\begin{definition}\label{def:radu-andre}
Let $\varphi\colon R \to S$ be a homomorphism of rings of prime characteristic $p > 0$. For every integer $e \ge 0$, consider the co-Cartesian diagram
  \[
    \begin{tikzcd}[column sep=4em]
      R \rar{F_R^e}\dar[swap]{\varphi} & F^e_{R*}R
      \arrow[bend left=30]{ddr}{F^e_{R*}\varphi}
      \dar{\varphi \otimes_R F^e_{R*}R}\\
      S \rar{\id_S \otimes_R F^e_R} \arrow[bend right=12,end
      anchor=west]{drr}[swap]{F^e_S} & S \otimes_R F^e_{R*}R
      \arrow[dashed]{dr}[description]{F^e_{S/R}}\\
      & & F^e_{S*}S
    \end{tikzcd}
  \]
in the category of rings. The \textsl{$e$-th Radu--Andr\'e ring} is the ring $S \otimes_R F^e_{R*}R$, and the \textsl{$e$-th relative Frobenius homomorphism} associated to $\varphi$ is the ring homomorphism
  \[
    \begin{tikzcd}[column sep=1.475em,row sep=0]
      \mathllap{F^e_{S/R}\colon} S \otimes_R F^e_{R*}R \rar & F^e_{S*}S\\
      s \otimes r \rar[mapsto] & s^{p^e} \varphi(r)\mathrlap{.}
    \end{tikzcd}
  \]
If $e = 1$, we denote $F^1_{S/R}$ by $F_{S/R}$. We also sometimes denote $F^e_{S/R}$ by $F^e_\varphi$.
\end{definition}

\begin{remark}
Even when $R$ and $S$ are Noetherian, the Radu--Andr\'e rings $S \otimes_R F^e_{R*}R$ are not necessarily Noetherian. For these rings to be Noetherian, it suffices, for example, for $R$ to be $F$-finite\footnote{Recall that a ring $R$ of prime characteristic $p > 0$ is \textsl{$F$-finite} if the Frobenius map $F_R$ is module-finite.}, or for $R \to S$ to be $F$-pure in the sense of Definition \ref{def:fpuremap} below \cite[Lem.\ 2.14]{Has10}. See \citeleft\citen{Rad92}\citemid Thm.\ 7\citepunct\citen{Dum96}\citemid Thm.\ 4.4\citepunct\citen{Has01}\citemid Lem.\ 4.2\citeright\ for more results on the Noetherianity of $S \otimes_R F^e_{R*}R$.
\end{remark}

Radu and Andr\'e used the homomorphism $F_{S/R}$ to give the following
characterization of regular homomorphisms, that is, flat ring maps with geometrically regular fibers. Note that setting $\varphi$ to be the homomorphism $\FF_p \to R$ in the statement below, one recovers Kunz's characterization of regular rings \cite[Thm.\ 2.1]{Kun69}.

\begin{citedthm}[{\citeleft\citen{Rad92}\citemid Thm.\
  4\citepunct\citen{And93}\citemid Thm.\ 1\citeright}]\label{thm:raduandre}
A homomorphism $\varphi\colon R \to S$ of Noetherian rings of prime characteristic $p > 0$ is regular if and only if $F_{S/R}$ is flat.
\end{citedthm}

\noindent Note that if the relative Frobenius homomorphism $F_{S/R}$ is flat, then it is automatically faithfully flat since $F_{S/R}$ induces a homeomorphism on spectra \cite[Exp.\ XV, Prop.\ 2$(a)$]{SGA5}.\medskip
\par One can weaken the faithful flatness of $F_{S/R}$ to obtain the following:

\begin{citeddef}[{\cite[(2.3) and Lem.\ 2.5.1]{Has10}}]\label{def:fpuremap}
A homomorphism $\varphi\colon R \to S$ of rings of prime characteristic
$p > 0$ is \textsl{$F$-pure} if $F_{S/R}$ is a pure ring homomorphism.
\end{citeddef}

\noindent While $F$-pure homomorphisms have geometrically $F$-pure fibers \cite[Cor.\ 2.16]{Has10}, the converse does not hold in general \cite[Rem.\ 2.17]{Has10}.

We end this section with a version of the relative Frobenius homomorphism for modules.

\begin{lemma}[{cf.\ \cite[p.\ 557]{Ene00}}]\label{lem:sillylemma}
Let $\varphi\colon R \to S$ be a homomorphism of Noetherian rings of
prime characteristic $p > 0$. Let $M$ be an $R$-module, and let $N_1$ and $N_2$ be $S$-modules equipped
with a homomorphism
  \[
    \psi\colon N_1 \longrightarrow F^e_{S*}N_2
  \]
of $S$-modules for some $e > 0$. Then, the homomorphism $\psi$ induces a homomorphism
  \[
    \begin{tikzcd}[row sep=0,column sep=1.475em]
      \widetilde{\psi}\colon N_1 \otimes_R F^e_{R*}M \rar & F^e_{S*}(N_2
      \otimes_R M)\\
      n \otimes m \rar[mapsto] & \psi(n) \otimes m
    \end{tikzcd}
  \]
of $(S,F^e_{R*}R)$-bimodules that is functorial in $M$.
\end{lemma}

\begin{proof}
We first note that $\psi$ induces an $R$-bilinear homomorphism
  \[
    \begin{tikzcd}[row sep=0,column sep=1.475em]
      N_1 \times F^e_{R*}M \rar & F^e_{S*}(N_2 \otimes_R M)\\
      (n,m) \rar[mapsto] & \psi(n) \otimes m
    \end{tikzcd}
  \]
where the left $R$-module structure on $F^e_{R*}M$ is given by $a \cdot m = a^{p^e}m$. Hence the homomorphism $\widetilde{\psi}$ exists as an $R$-module homomorphism by the universal property of tensor products.
This homomorphism $\widetilde{\psi}$ is in fact $(S,F^e_{R*}R)$-bilinear since
  \begin{gather*}
    \widetilde{\psi}(n \otimes m) \cdot a = (\psi(n) \otimes m)\cdot a = \psi(n)
    \otimes ma = \widetilde{\psi}(n \otimes ma) = \widetilde{\psi}\bigl((n
    \otimes m) \cdot a\bigr)\\
    b \cdot \widetilde{\psi}(n \otimes m) = b \cdot(\psi(n) \otimes m) =
    b^{p^e}\psi(n) \otimes m = \psi(bn) \otimes m = \widetilde{\psi}(bn \otimes
    m) = \widetilde{\psi}\bigl(b\cdot(n \otimes m)\bigr)
  \end{gather*}
for all $a \in R$ and for all $b \in S$. Functoriality in $M$ follows from the construction of $\widetilde{\psi}$ using the universal property of tensor products.
\end{proof}

\begin{remark}
The construction in Lemma \ref{lem:sillylemma} yields the relative Frobenius $F_{S/R}$ when we apply it to $M = R$, $N_1 = N_2 = S$ and $\psi = F_S\colon S \rightarrow F_{S*}S$.
\end{remark}

\section{Localization and descent under faithfully flat
homomorphisms}\label{sect:locanddescent}

In this section, we prove assertions (\ref{property:localization}) and
(\ref{property:descent}) for $F$-injectivity, namely, that $F$-injectivity localizes (Proposition \ref{prop:finjlocalizes}) and descends along faithfully flat homomorphisms (Theorem \ref{thm:descentF-injective}). We then discuss the behavior of geometric $F$-injectivity under infinite purely inseparable extensions in \S\ref{subsect:geomfinjinfiniteext}.

\subsection{Characterizations of \emph{F}-injectivity using module-finite algebras}
An important ingredient in proving localization and descent of $F$-injectivity is the following alternative characterization of $F$-injectivity in terms of module-finite algebras over $R$, which was pointed out to us by Karl Schwede.

\begin{lemma}\label{lem:finjmodfin}
Let $(R,\fm)$ be a Noetherian local ring of prime characteristic $p > 0$. Fix a filtered direct system $\{\psi_\alpha\colon R \to S_\alpha\}_\alpha$ of module-finite homomorphisms such that $F_R\colon R \to F_{R*}R$ is the colimit of this directed system. Then, the ring $R$ is $F$-injective if and only if the $R$-module homomorphisms
  \begin{equation}\label{eq:finjfinalg}
    H^i_\fm(\psi_\alpha)\colon H^i_\fm(R) \longrightarrow H^i_\fm(S_\alpha)
  \end{equation}
are injective for all $i$ and for all $\alpha$.
\end{lemma}

\begin{proof}
The homomorphism $H^i_\fm(F_R)\colon H^i_\fm(R) \to H^i_\fm(F_{R*}R)$ factors as
\begin{equation}\label{eq:finjfactor}
    H^i_\fm(R) \xrightarrow{H^i_\fm(\psi_\alpha)} H^i_\fm(S_\alpha)
    \longrightarrow H^i_\fm(F_*R)
  \end{equation}
for every $\alpha$. Thus, for the direction $\Rightarrow$, it suffices to note that if the composition \eqref{eq:finjfactor} is injective, then the first homomorphism $H^i_\fm(\psi_\alpha)$ is also injective.
For the direction $\Leftarrow$, it suffices to take colimits in
\eqref{eq:finjfactor}, since local cohomology commutes with filtered colimits.
\end{proof}

We will also use the following characterization of $F$-injectivity in terms of homomorphisms on $\Ext$ modules. Under $F$-finiteness hypotheses, this characterization is due to Fedder \cite[Rem.\ on p.\ 473]{Fed83} for Cohen--Macaulay rings, and is implicit in
the proof of \cite[Prop.~4.3]{Sch09}.

\begin{lemma}[cf.\ {\cite[Lem.\ A.1]{Mur}}]\label{lem:finjdc}
Let $R$ be a Noetherian ring of prime characteristic $p > 0$ with a dualizing complex $\omega_R^\bullet$. Fix a directed system $\{\psi_\alpha\colon R \to S_\alpha\}_\alpha$ of module-finite homomorphisms such that $F_R\colon R \to F_{R*}R$ is the colimit of this directed system. Then, the ring $R$ is $F$-injective if and only if the $R$-module homomorphisms
  \begin{equation}\label{eq:grottracefinj}
    \psi_\alpha^*\colon
    \Ext^{-i}_R(\psi_{\alpha*}S_\alpha,\omega_R^\bullet)
    \longrightarrow \Ext^{-i}_R(R,\omega_R^\bullet)
  \end{equation}
are surjective for all $i$ and for all $\alpha$.
\end{lemma}

\begin{proof}
  This follows from Lemma \ref{lem:finjmodfin} and 
  Grothendieck local duality \cite[Cor.\ V.6.3]{Har66}, since
  Ext modules \cite[Prop.\ X.6.10$(b)$]{BouA} (here is where we use the
  module-finiteness of the $S_\alpha$) and dualizing
  complexes \cite[Cor.\ V.2.3]{Har66} are compatible with localization.
\end{proof}

\subsection{Localization of \emph{F}-injectivity}
In \cite[Def.\ 4.4]{Sch09}, Schwede defines a Noetherian ring $R$ to be
$F$-injective if $R_\fp$ is $F$-injective for every prime ideal $\fp \subseteq R$. We show that Schwede's definition is equivalent to our definition of $F$-injectivity (Definition \ref{def:FI}). This result was shown by Schwede under the additional assumption that $R$ is $F$-finite \cite[Prop.\ 4.3]{Sch09}, and by Hashimoto for the CMFI property \citeleft\citen{Has01}\citemid p.\ 238\citepunct \citen{Has10}\citemid Cor.\ 4.11\citeright.

\begin{proposition}\label{prop:finjlocalizes}
Let $R$ be a Noetherian ring of prime characteristic $p > 0$. The ring $R$ is $F$-injective if and only if $R_\fp$ is $F$-injective for every prime ideal $\fp \subseteq R$. In particular, if $R$ is $F$-injective, then for every multiplicative set $S \subseteq R$, the localization $S^{-1}R$ is $F$-injective.
\end{proposition}

\begin{proof}
We first prove the if and only if statement.
The direction $\Leftarrow$ is true by definition, and hence it suffices to show the direction $\Rightarrow$. We claim that it suffices to consider the case when $R$ is a complete local ring. Suppose $R$ is $F$-injective, and let $\fp \subseteq R$ be a prime ideal.
Let $\fm \subseteq R$ be a maximal ideal containing $\fp$, and consider a prime
ideal $\fq \subseteq \widehat{R_\fm}$ minimal over
$\fp\widehat{R_\fm}$.
Then, the local homomorphism
  \[
    R_\fp \longrightarrow (\widehat{R_\fm})_\fq
  \]
is faithfully flat with zero-dimensional closed fiber.
Since $\widehat{R_\fm}$ is $F$-injective by \cite[Rem.\ on p.\
473]{Fed83}, the localization $(\widehat{R_\fm})_\fq$ is also by the Proposition
in the complete local case,
and hence $R_\fp$ is $F$-injective by \cite[Lem.\ A.3]{Mur}.

\par It remains to show the direction $\Rightarrow$ when $R$ is complete local,
in which case $R$ has a dualizing complex \cite[(4) on p.\ 299]{Har66}.
Write $F_R\colon R \to F_{R*}R$ as a filtered colimit of module-finite homomorphisms $\psi_\alpha\colon R \to S_\alpha$.
The statement now follows from Lemma \ref{lem:finjdc}, since the surjectivity of
the homomorphisms \eqref{eq:grottracefinj} localizes by \cite[Prop.\
X.6.10$(b)$]{BouA}.
  
The last localization statement now follows because the local rings of $S^{-1}R$ coincide with the local rings of $R$ at the primes of $R$ that do not intersect $S$.  
\end{proof}

\begin{remark}\label{rem:gammaF-injec}
As pointed out by Linquan Ma, one can also prove the direction $\Rightarrow$ in Proposition \ref{prop:finjlocalizes} in the complete local case using the gamma construction of Hochster--Huneke \cite[(6.11)]{HH94}. By the gamma construction and \cite[Thm.\ $3.4(ii)$]{Mur} (see also Remark \ref{rem:gamma3forfinj}), there exists a faithfully flat extension $R \to R^\Gamma$, such that $R^\Gamma$ is $F$-finite, and such that the induced morphism $\Spec(R^\Gamma) \to \Spec(R)$ is a homeomorphism that identifies $F$-injective loci. Since the $F$-injective locus on $\Spec(R^\Gamma)$ is stable under generization by \cite[Prop.\ 4.3]{Sch09}, the $F$-injective locus on $\Spec(R)$ is also stable under generization.
\end{remark}

As a consequence of Proposition \ref{prop:finjlocalizes}, the following properties of $F$-injective Noetherian rings follow without any $F$-finiteness hypotheses.
See, e.g., \citeleft\citen{Swa80}\citepunct \citen{Yan85}\citeright\ for the
notions of seminormality and weak normality in
$(\ref{cor:finjwnitem})$.

\begin{corollary}\label{cor:finjwn}
Let $R$ be a Noetherian, $F$-injective ring of prime characteristic $p > 0$. We then have the following:
  \begin{enumerate}[label=$(\roman*)$,ref=\roman*]
    \item $R$ is reduced.\label{cor:finjred}
    \item $R$ is approximately Gorenstein.\label{cor:finjapproxgor}
    \item $R$ is weakly normal, and in particular, seminormal.
      \label{cor:finjwnitem}
    \item If $R$ is in addition local and $I$ is an ideal generated by a regular sequence, then the map
      \[
        \id_{R/I} \otimes_R F^e_R \colon R/I \longrightarrow R/I \otimes_R F^e_{R*}(R) 
      \]
is injective for all $e > 0$. In particular, $I$ is Frobenius closed.\label{cor:finjfrobclosed}
  \end{enumerate}

\end{corollary}
\begin{proof}
  For $(\ref{cor:finjred})$ and $(\ref{cor:finjfrobclosed})$, the proofs in
  \cite[Rem.\ 2.6]{SZ13} and \cite[Prop.\ 3.11]{QS17} in the $F$-finite case
  apply, since they only use $F$-finiteness to conclude that $F$-injectivity
  localizes (Proposition \ref{prop:finjlocalizes}).
$(\ref{cor:finjred})$ is also shown in \cite[Lem.\ 3.11]{QS17} using the gamma
construction.
  
\par For $(\ref{cor:finjapproxgor})$, we may assume that $R$ is local since the property of being approximately Gorenstein is local
\cite[Def.\ 1.3]{Hoc77}.
Since $\widehat{R}$ is $F$-injective \cite[Rem.\ on p.\ 473]{Fed83}, it is
reduced by $(\ref{cor:finjred})$, and hence approximately Gorenstein by
\cite[Thm.\ 1.7]{Hoc77}.
Thus, $R$ is approximately Gorenstein by \cite[Cor.\ 2.2]{Hoc77}.
  
\par $(\ref{cor:finjwnitem})$ is shown in the $F$-finite case by Schwede \cite[Thm.\ 4.7]{Sch09}.
The same proof applies once we know that $F$-injectivity localizes (Proposition
\ref{prop:finjlocalizes}) and that $(\ref{cor:finjred})$ holds.
\end{proof}
  
\begin{remark}
As a consequence of Corollary \ref{cor:finjwn}$(\ref{cor:finjfrobclosed})$, the proof of \cite[Cor.\ 3.14]{QS17} yields the following: if $(R,\fm)$ is a Noetherian local ring, and $I = (x_1,x_2,\dots,x_t)$ is an ideal generated by a regular sequence such that $R/x_1R$ is $F$-injective, then the Frobenius actions on $H^t_I(R)$ and $H^t_\fm(R)$ are injective. In particular, this shows that $F$-injectivity deforms when $R$ is Cohen--Macaulay (the latter was first shown in \cite[Thm.\ 3.4(1)]{Fed83}).
\end{remark}

Globally, localization of $F$-injectivity implies that $F$-injectivity can be checked on an affine open cover.

\begin{corollary}\label{cor:globalfinj}
Let $X$ be a locally Noetherian scheme of prime characteristic $p > 0$. Then the following are equivalent:
\begin{enumerate}[label=$(\roman*)$,ref=\roman*]
	\item For every affine open subscheme $\Spec(R)$ of $X$, the ring $R$ is $F$-injective.\label{cor:everyaffinefinj}
  \item $X$ is $F$-injective.\label{cor:xisfinj}
	\item For every locally closed point $x \in X$, the stalk $\mathcal{O}_{X,x}$ is $F$-injective.\label{cor:stalksarefinj}
\end{enumerate}
Furthermore, if $X$ is locally of finite type over a field $k$ of characteristic $p > 0$, then \emph{$(\ref{cor:everyaffinefinj})$--$(\ref{cor:stalksarefinj})$} are equivalent to $\mathcal{O}_{X,x}$ being $F$-injective for every closed point $x \in X$.
\end{corollary}

\begin{proof}
Clearly $(\ref{cor:everyaffinefinj}) \Rightarrow (\ref{cor:xisfinj})$. If $X$ has an affine open cover $\{R_\alpha\}_\alpha$ such that every $R_\alpha$ is $F$-injective, then Proposition \ref{prop:finjlocalizes} shows that all the stalks of $\mathcal{O}_X$ are $F$-injective local rings. Thus, $(\ref{cor:xisfinj}) \Rightarrow (\ref{cor:stalksarefinj})$. Finally, $(\ref{cor:stalksarefinj}) \Rightarrow (\ref{cor:everyaffinefinj})$ follows because if $\Spec(R)$ is an affine open subscheme of $X$, then any closed point of $\Spec(R)$ is a locally closed point of $X$.

If $X$ is in addition locally of finite type over a field $k$, then the set of closed points of $X$ coincide with the set of locally closed points of $X$, thus proving the equivalence of the last assertion in the statement of the Corollary with $(\ref{cor:everyaffinefinj})$--$(\ref{cor:stalksarefinj})$.
\end{proof}

\subsection{Descent of \emph{F}-injectivity}
We now show that $F$-injectivity descends along faithfully flat
homomorphisms, thereby establishing (\ref{property:descent}) for the property $\cP$ of being $F$-injective. The corresponding result for the CMFI property was shown by Hashimoto \citeleft\citen{Has01}\citemid Lem.\ 5.2\citepunct\citen{Has10}\citemid Lem.\ 4.6\citeright. 

\begin{theorem}\label{thm:descentF-injective}
Let $\varphi\colon R \to S$ be a homomorphism of Noetherian rings of prime characteristic $p > 0$ that is surjective on spectra satisfying the following property: for all prime ideals $\fq \in \Spec(S)$, the induced map $R_{\varphi^{-1}(\fq)} \rightarrow S_\fq$ is pure (for example if $\varphi$ is faithfully flat). Then if $S$ is $F$-injective, so is $R$.
\end{theorem}

\begin{proof}
Let $\fp \subseteq R$ be a maximal ideal, and let $\fq \subseteq S$ be a prime ideal that is minimal among primes of $S$ lying over $\fp$ (such a $\fq$ exists because $\varphi$ is surjective on spectra). By Proposition \ref{prop:finjlocalizes}, we know that $S_\fq$ is $F$-injective. Since the local homomorphism $R_\fp \to S_\fq$ is pure with zero-dimensional closed fiber, we have reduced to the case shown in \cite[Lem.\ A.3]{Mur}.
\end{proof}

Our hypothesis on $\varphi$ in Theorem \ref{thm:descentF-injective} implies that $\varphi$ is pure. This is because purity is a local condition, and if we choose a prime $\fq$ of $S$ lying over an arbitrary prime $\fp$ of $R$, then purity of the composition $R_\fp \rightarrow S_\fp \rightarrow S_\fq$ implies $R_\fp \rightarrow S_\fp$ is also pure. 
Maps $\varphi$ satisfying the hypotheses of Theorem \ref{thm:descentF-injective} are also called \textsl{strongly pure} in \cite[p.\ 38]{CGM16}, and there exist examples of pure maps that are not strongly pure \cite[Cor.\ 5.6.2]{CGM16}.
\par For pure maps that are not necessarily strongly pure, we have the following
descent result.
This result is optimal given known counterexamples to descent of $F$-injectivity
under pure (or even split) maps \citeleft\citen{Wat97}\citemid Ex.\ 3.3(1)\citepunct
\citen{Ngu12}\citemid Ex.\ 6.6 and Rem.\ 6.7\citeright\ that are
not quasi-finite.

\begin{proposition}
\label{prop:quasi-finite-pure-descent}
Let $\varphi\colon R \to S$ be a quasi-finite (e.g., module-finite), pure homomorphism
  of Noetherian rings of prime characteristic $p > 0$.
  If $S$ is $F$-injective, so is $R$.
\end{proposition}
\begin{proof}
  Let $\fm \subseteq R$ be a maximal ideal.
  Since quasi-finiteness and purity are preserved under base change and $F$-injectivity 
  of $R$ can be checked at its maximal ideals, replacing $\varphi\colon R \to S$ 
  with its localization $\varphi_\fm\colon R_\fm \to S \otimes_R R_\fm$, it suffices 
  to consider the case when $R$ is local.
  \par By \cite[Lem.\ 2.2]{HH95}, there exists a maximal ideal $\fq \subseteq S$
  lying over $\fm$ such that $R \to S_\fq$ is pure.
  Since $\varphi$ is quasi-finite, we also have $\sqrt{\fm S_\fq} = \fq$.
  The statement now follows from \cite[Lem.\ A.3]{Mur}.
\end{proof}

\subsection{Geometrically CMFI rings and infinite purely inseparable extensions}\label{subsect:geomfinjinfiniteext}

In the proof of Theorem \ref{thm:finjascendscmfi}, we need to consider how geometrically CMFI rings interact with possibly infinite purely inseparable extensions of the base field.

\begin{proposition}[cf.\ {\cite[Prop.\ 2.21]{Ene00}}]\label{prop:geomfinjpurelyinsep}
Let $(R,\fm)$ be a geometrically CMFI local $k$-algebra, where $k$ is a field of characteristic $p > 0$. Let $k \subseteq k'$ be a purely inseparable extension (not necessarily finite). If $I$ is an ideal of $R$ generated by a system of parameters, then $I(R \otimes_k k')$ is Frobenius closed in $R \otimes_k k'$. In particular, if $R \otimes_k k'$ is Noetherian, then $R \otimes_k k'$ is CMFI.
\end{proposition}

\begin{proof}
Write $k' = \bigcup_\alpha L_\alpha$ as the directed union of finite purely
inseparable subextensions of $k \subseteq k'$.
Then, $R \otimes_k k' = \bigcup_\alpha R \otimes_k L_\alpha$.
Since $R \hookrightarrow R \otimes_k L_\alpha$ is finite and purely inseparable,
the ideal $I(R \otimes_k L_\alpha)$ is generated by a system of parameters.
Moreover, since $R$ is geometrically CMFI, Lemma \ref{lem:cmfiparam} implies the
ideal $I(R \otimes_k L_\alpha)$ is Frobenius closed.
We then have
\[
\bigl(I(R\otimes_k k')\bigr)^{F} = \bigcup_{\alpha} \bigl(I(R \otimes_k L_\alpha)\bigr)^{F} = \bigcup_{\alpha}I(R \otimes_k L_\alpha) = I(R \otimes_k k').
\]
Here the second and third equalities are straightforward to verify, while the
first equality follows from the equality
\[
  I^{[p^e]}(R \otimes_k k') = \bigcup_\alpha I^{[p^e]}(R \otimes_k L_\alpha).
\]
This shows that $I(R \otimes_k k')$ is Frobenius closed, proving the first assertion of the Proposition.

For the second assertion, suppose $R \otimes_k k'$ is Noetherian.
Since $R \hookrightarrow R \otimes_k k'$ is purely inseparable, it is a flat
local homomorphism of local rings, and $R \otimes_k k'$ is Cohen--Macaulay by
\cite[Cor.\ to Thm.\ 23.3]{Mat89}.
The fact that $R \otimes_k k'$ is $F$-injective then follows by Lemma \ref{lem:cmfiparam} because
$I(R \otimes_k k')$ is a Frobenius closed ideal generated by a system of
parameters.
\end{proof}

\section{Ascent under flat homomorphisms with geometrically CMFI fibers}\label{sect:ascent}
In \cite{Vel95}, V\'elez showed the following base change result for
$F$-rationality:

\begin{citedthm}[{\cite[Thm.\ 3.1]{Vel95}}]\label{thm:velez}
Let $\varphi\colon R \to S$ be a regular homomorphism of locally excellent Noetherian rings of prime characteristic $p > 0$. If $R$ is $F$-rational, then $S$ is also $F$-rational.
\end{citedthm}

Enescu \cite[Thm.\ 2.27]{Ene00} and Hashimoto \cite[Thm.\ 6.4]{Has01} proved that in fact, it suffices to assume that $\varphi$ has
geometrically $F$-rational fibers.\footnote{Enescu assumes that $B/\fm B
\otimes_{A/\fm} F^e_*(A/\fm)$ is Noetherian for every $e > 0$, which is used
in the proof of \cite[Thm.\ 2.19]{Ene00}.
This latter result holds without this Noetherianity assumption by applying
Proposition \ref{prop:geomfinjpurelyinsep} to the system of parameters
$y_1,y_2,\ldots,y_d$ on $B/\fm B$ (cf.\ \cite[Rem.\ 2.20]{Ene00}).}
Our goal is to show the following analogue of their results for
$F$-injectivity.

\begin{customthm}{A}\label{thm:finjascendscmfi}
Let $\varphi\colon (R,\fm) \to (S,\fn)$ be a flat local homomorphism of
Noetherian local rings of prime characteristic $p > 0$ whose closed fiber $S/\fm S$ is geometrically CMFI over $R/\fm$. If $R$ is $F$-injective, then $S$ is $F$-injective.
\end{customthm}

Since $F$-injectivity localizes (Proposition \ref{prop:finjlocalizes}), we obtain the following non-local result as a consequence:

\begin{corollary}\label{cor:finjascendscmfinonlocal}
Let $\varphi\colon R \to S$ be a flat homomorphism of Noetherian rings of prime characteristic $p > 0$ whose fibers are geometrically CMFI. If $R$ is $F$-injective, then $S$ is $F$-injective.
\end{corollary}

\begin{remark}
If $\varphi\colon R \to S$ is a flat homomorphism of Noetherian rings of prime characteristic $p > 0$ such that $\varphi$ contracts maximal ideals of $S$ to maximal ideals of $R$ (this holds, for instance, if $\varphi$ is  of finite type and $R$ and $S$ are Jacobson rings \cite[Thm.\ V.3.3]{BouCA}), then $F$-injectivity ascends under the weaker hypotheses that $R$ is $F$-injective and \emph{only} the closed fibers of $\varphi$ are geometrically CMFI. This is because $F$-injectivity of $S$ is checked at the maximal ideals of $S$.
\end{remark}

Theorem \ref{thm:finjascendscmfi} also implies $F$-injectivity is preserved under (strict) Henselization.

\begin{corollary}
\label{cor:henselization-F-inj}
If $(R,\fm)$ is a Noetherian local ring of prime characteristic $p > 0$ that is $F$-injective, then its Henselization $R^h$ and strict Henselization $R^{sh}$ are also $F$-injective.
\end{corollary}

\begin{proof}
The closed fiber of $R \rightarrow R^h$ is an isomorphism \cite[Thm.\
18.6.6$(iii)$]{EGAIV4}, and the closed fiber of $R \rightarrow R^{sh}$ is a separable
algebraic field extension by construction.
In either case, the closed fiber is geometrically CMFI, and Theorem
\ref{thm:finjascendscmfi} implies $R^h$ and
$R^{sh}$ are $F$-injective.
\end{proof}

\par
Theorem \ref{thm:finjascendscmfi} is known to fail without geometric assumptions, even if the fibers are regular (see \cite{SZ13} and \cite[\S4]{Ene09}). Theorem \ref{thm:finjascendscmfi} extends similar base change results due to Hashimoto \cite[Cor.\ 5.7]{Has01} and Aberbach--Enescu \cite[Thm.\ 4.3]{Ene09}, which assume that $R$ is Cohen--Macaulay.

\medskip
\par The key ingredient to proving Theorem \ref{thm:finjascendscmfi} is the injectivity of the action of relative Frobenius on local cohomology, which we study in \S\ref{subsect:raduandrepure} after reviewing some preliminaries on local cohomology in \S\ref{subsect:localcoh}. We then prove Theorem \ref{thm:finjascendscmfi}, which we use to analyze the behavior of geometric $F$-injectivity under arbitrary finitely generated field extensions in \S\ref{subsect:geomfinjarbext}.

\subsection{Preliminaries on local cohomology}\label{subsect:localcoh}
\par In order to prove Theorem \ref{thm:finjascendscmfi}, we will need the following preliminary results on local cohomology and flat base change.

\begin{lemma}\label{lem:ascentkeyiso}
Let $\varphi\colon (R, \fm) \rightarrow (S, \fn)$ be a flat local homomorphism of Noetherian rings and let $I$ be an ideal of $R$.
Let $J = (y_1,y_2,\ldots,y_n) \subseteq \fn$ be an ideal of $S$ such that the images 
of $y_1,y_2,\ldots,y_n$ in $S/{\fm}S$ form a regular sequence on $S/\fm{S}$.
\begin{enumerate}[label=$(\roman*)$,ref=\roman*]
  \item\label{lem:ascentkeyisoflatlc}
    The sequence $y_1,y_2,\dots,y_n$ forms a regular sequence on $S$.
    Moreover, the modules $S/(y_1^t,y_2^t,\ldots,y_n^t)$ and $H^n_J(S)$ are flat
    over $R$ for every integer $t > 0$.

\item
  For every $R$-module $M$, there is an isomorphism
  \begin{equation}\label{eq:ascentkeyiso3}
    H^i_{IS + J}(S \otimes_R M) \simeq H^n_J(S) \otimes_R H^{i-n}_I(M)
  \end{equation}
of $S$-modules that is functorial in $M$. \label{lem:ascentkeyiso3}
\item\label{lem:ascentkeyiso3withfrob}
  Suppose that 
  $R$ and $S$ are of prime characteristic $p > 0$.
  Then, the isomorphism \eqref{eq:ascentkeyiso3} for $M = F^e_{R*}R$ is
  compatible with the relative Frobenius homomorphism, i.e., the diagram
  \[
    \begin{tikzcd}[column sep=scriptsize,row sep=small]
      H^i_{IS + J}(S \otimes_R F^e_{R*}R) \rar{H^i_{IS+J}(F^e_{S/R})}
      \arrow[phantom]{d}[sloped]{\simeq}&
      H^i_{IS + J}(F^e_{S*}S) \arrow[phantom]{d}[sloped]{\simeq}\\
      H^n_J(S) \otimes_R H^{i-n}_I(F^e_{R*}R) \rar & 
      F^e_{S*}\bigl(H^n_J(S) \otimes_R H^{i-n}_I(R)\bigr)
    \end{tikzcd}
  \]
  commutes, where the bottom horizontal homomorphism is that induced by the
  Frobenius action $H^i_{J}(F^e_S)$ on the local cohomology module
  $H^n_J(S)$ as in Lemma \ref{lem:sillylemma}.
\end{enumerate}
\end{lemma}
\begin{proof}
\par For $(\ref{lem:ascentkeyisoflatlc})$, the flatness of $H^n_J(S)$
follows from the flatness of $S/(y_1^t,y_2^t,\ldots,y_n^t)$, since the former is
a filtered colimit of modules of the latter form.
It therefore suffices to note that for every integer $t > 0$, the
images of $y_1^t,y_2^t,\dots, y_n^t$ in $S/\fm S$
remain a regular sequence on $S/\fm S$, after which the rest of
$(\ref{lem:ascentkeyisoflatlc})$ follows from
\cite[Lem.\ 7.10$(b)$]{HH94}.

We now show $(\ref{lem:ascentkeyiso3})$.
Let $x_1,x_2,\ldots,x_m \in R$ be a set of generators for $I$.
We will denote by $\mathbf{x}$ and $\mathbf{y}$ the $m$- and $n$-tuples
$(x_1,x_2,\ldots,x_m)$ and $(y_1,y_2,\ldots,y_n)$, respectively.
We then have the following chain of isomorphisms and quasi-isomorphisms of
complexes:
\begin{equation}\label{eq:ascentkeyiso3chain}
  \begin{aligned}
    \check{C}^\bullet(\mathbf{y},\mathbf{x};S \otimes_R M)
    &\overset{\sim}{\longleftarrow}
    \check{C}^\bullet(\mathbf{y};S) \otimes_S S \otimes_R
    \check{C}^\bullet(\mathbf{x};R) \otimes_R M\\
    &\overset{\sim}{\longrightarrow}
    \check{C}^\bullet(\mathbf{y};S) \otimes_R
    \check{C}^\bullet(\mathbf{x};M)\\
    &\overset{\text{qis}}{\longrightarrow}
    H^n_J(S)[-n] \otimes_R \check{C}^\bullet(\mathbf{x};M).
\end{aligned}
\end{equation}
Here, the first two isomorphisms follow from the definition of the
\v{C}ech complex \cite[Def.\ 6.26]{twentyfour}.
The last quasi-isomorphism follows from $(\ref{lem:ascentkeyisoflatlc})$:
The \v{C}ech complex $\check{C}^\bullet(\mathbf{y};S)$ is a complex of flat
$R$-modules quasi-isomorphic to the complex $H^n_J(S)[-n]$ of flat $R$-modules by the fact
that $y_1,y_2,\dots,y_n$ is a regular sequence on $S$. Note that the
 quasi-isomorphism $\check{C}^\bullet(\mathbf{y};S) \rightarrow H^n_J(S)[-n]$
remains a quasi-isomorphism after tensoring by $\check{C}^\bullet(\mathbf{x};M)$
by \cite[Cor.\ 5 to
Thm.\ X.4.3]{BouA}.
By \cite[Exp.\ II, Prop.\ 5]{SGA2}, applying cohomology in
\eqref{eq:ascentkeyiso3chain}
yields the desired isomorphism \eqref{eq:ascentkeyiso3}, since all involved
(quasi-)isomorphisms are functorial in $M$.
\par We now show $(\ref{lem:ascentkeyiso3withfrob})$.
Setting $M = F^e_{R*}R$ in the situation of $(\ref{lem:ascentkeyiso3})$, we
trace the relative Frobenius homomorphism $F^e_{S/R}$ through the chain of
(quasi-)isomorphisms \eqref{eq:ascentkeyiso3chain}:
\[
  \begin{tikzcd}
    \check{C}^\bullet(\mathbf{y},\mathbf{x};S \otimes_R F^e_*R)
    \rar{\check{C}^\bullet(\mathbf{y},\mathbf{x};F^e_{S/R})}
    \arrow[leftarrow]{d}[sloped,above]{\sim}
    & \check{C}^\bullet(\mathbf{y},\mathbf{x};F^e_{S*}S)
    \arrow[leftarrow]{d}[sloped,above]{\sim}\\
    \bigl(
    \check{C}^\bullet(\mathbf{y};S) \otimes_S S \otimes_R
    \check{C}^\bullet(\mathbf{x};R)\bigr) \otimes_R F^e_{*R}R
    \rar
    \arrow{d}[sloped,above]{\sim}
    & F^e_{S*}\bigl(
    \check{C}^\bullet(\mathbf{y};S) \otimes_S S \otimes_R
    \check{C}^\bullet(\mathbf{x};R)
    \bigr)
    \arrow{d}[sloped,above]{\sim}\\
    \check{C}^\bullet(\mathbf{y};S) \otimes_R
    F^e_{*R}\check{C}^\bullet(\mathbf{x};R)
    \rar
    \arrow{d}{\text{qis}}
    & F^e_{S*}\bigl(
    \check{C}^\bullet(\mathbf{y};S) \otimes_R
    \check{C}^\bullet(\mathbf{x};R)
    \bigr)
    \arrow{d}{\text{qis}}\\
    H^n_J(S)[-n] \otimes_R F^e_{*R}\check{C}^\bullet(\mathbf{x};R)
    \rar
    & F^e_{S*}\bigl( H^n_J(S)[-n] \otimes_R
    \check{C}^\bullet(\mathbf{x};R)\bigr)\mathrlap{.}
  \end{tikzcd}
\]
The last two horizontal maps in the above diagram are
induced by the homomorphism of
modules in Lemma \ref{lem:sillylemma} via the
Frobenius action on each term of the complex $\check{C}^\bullet(\mathbf{y};S)$
and on $H^n_J(S)$, respectively.
To ensure the diagram commutes, one can trace elements of the individual terms of
the complexes through the specific
description of the homomorphisms involved.
Applying cohomology gives the statement in
$(\ref{lem:ascentkeyiso3withfrob})$.
\end{proof}

\subsection{Proof of Theorem \ref{thm:finjascendscmfi}}
\label{subsect:raduandrepure}
The key technical result used in the proof of Theorem \ref{thm:finjascendscmfi} is the following:

\begin{proposition}\label{prop:ene219}
Let $\varphi\colon (R,\fm) \to (S,\fn)$ be a flat local homomorphism of
Noetherian local rings of prime characteristic $p > 0$ whose closed fiber $S/\fm S$ is geometrically CMFI over $R/\fm$. Suppose $J = (y_1,y_2,\ldots,y_n)$ is an ideal of $S$ such that the images of $y_1,y_2,\ldots,y_n$ in $S/\fm S$ form a system of parameters in $S/\fm S$. For every Artinian $R$-module $M$ and for every $e >0$, the homomorphism
\[
  \begin{tikzcd}[nodes={execute at begin node=\everymath{\displaystyle}},row sep=0,column sep=1.475em]
    \frac{S}{J} \otimes_R F^e_{R*}M \rar & F^e_{S*}\biggl(
    \frac{S}{J^{[p^e]}} \otimes_R M \biggr)\\
    (s + J) \otimes m \rar[mapsto] & (s^{p^e} + J^{[p^e]}) \otimes m
\end{tikzcd}
\]
from Lemma \ref{lem:sillylemma} is injective.
\end{proposition}

\begin{proof}
We first reduce to the case when $M$ is of finite length.
Write $M$ as the union $\bigcup_\alpha M_\alpha$ of its finitely generated $R$-submodules $M_\alpha$.
Since $M$ is Artinian, the $R$-modules $M_\alpha$ are of finite length.
The statement for $M$ then follows from the statement for the $M_\alpha$ by
taking filtered colimits.

\par We now assume $M$ is of finite length. We proceed by induction on the length of $M$. If $M$ has length $1$, then $M \simeq R/\fm \eqqcolon k$ is the residue field of $R$. It therefore suffices to show 
\begin{equation}\label{eq:thishomomorphism}
  \begin{tikzcd}[nodes={execute at begin node=\everymath{\displaystyle}},
    row sep=0,column sep=1.475em]
    \frac{S}{J} \otimes_R F^e_{R*}k \rar &
    F^e_{S*}\biggl( \frac{S}{J^{[p^e]}} \otimes_R k \biggr)\\
    s \otimes r \rar[mapsto] & s^{p^e} \otimes r
  \end{tikzcd}
\end{equation}
is injective for every positive integer $e$.
Under the composition of isomorphisms
\[
  \frac{S}{J^{[p^f]}} \otimes_R (-) \simeq \frac{S}{J^{[p^f]}} \otimes_R
  \frac{R}{\fm} \otimes_k (-) \simeq \frac{S}{\fm S+J^{[p^f]}} \otimes_k (-)
\]
of functors on the category of $R$-modules annihilated by the maximal ideal
$\fm$ for both $f = 0$ and $f = e$,
the homomorphism \eqref{eq:thishomomorphism} can be identified with the first homomorphism in the composition
\begin{equation}\label{eq:cmficlosedfiber}
\mathclap{\begin{tikzcd}[nodes={execute at begin node=\everymath{\displaystyle}}, row sep=0,column sep=small,ampersand replacement=\&]
\frac{S}{\fm S+J} \otimes_k F^e_{k*}k \rar \&
  F^e_{(S/\fm S)*}\Biggl( \frac{S}{\fm S+J^{[p^e]}} \otimes_k k\Biggr)
      \rar \&
      F^e_{(S/\fm S)*}\Biggl( \frac{S}{\fm S+J^{[p^e]}} \otimes_k
      F^e_{k*}k\Biggr)\\
      s \otimes r \rar[mapsto] \& s^{p^e} \otimes r\\
      \& \tilde{s} \otimes \tilde{r} \rar[mapsto] \& \tilde{s} \otimes
      \tilde{r}^{p^e}\mathrlap{.}
  \end{tikzcd}}
\end{equation}
Since $S/\fm S$ is geometrically CMFI over $R/\fm$, and the ideal $J(S/\fm S)$ is Frobenius closed in $S/\fm S$ by Lemma \ref{lem:cmfiparam}, we then get by Proposition \ref{prop:geomfinjpurelyinsep} that $J(S/\fm S \otimes_k F^e_{k*}k)$ is Frobenius closed in $S/\fm S \otimes_k F^e_{k*}k \simeq S/\fm S \otimes_k k^{1/p^e}$. Since 
\[
\frac{S}{\fm S+J} \otimes_k F^e_{k*}k \simeq \frac{S}{J} \otimes_{S} \bigg{(}\frac{S}{\fm S} \otimes_k F^e_{k*}k \bigg{)} \simeq \frac{S/\fm S \otimes_k F^e_{k*}k}{J(S/\fm S \otimes_k F^e_{k*}k)},
\]
and similarly,
\[
F^e_{(S/\fm S)*}\Biggl( \frac{S}{\fm S+J^{[p^e]}} \otimes_k F^e_{k*}k\Biggr) \simeq F^e_{(S/\fm S)*}\bigg{(}\frac{S/\fm S \otimes_k F^e_{k*}k}{J^{[p^e]}(S/\fm S \otimes_k F^e_{k*}k)}\bigg{)},
\]
it follows that the composition in \eqref{eq:cmficlosedfiber} is injective for all $e > 0$. Therefore in particular, the first homomorphism must be injective.
    
\par It remains to prove the inductive step, in which case there exist two $R$-modules $M_1,M_2$ of length strictly less than that of $M$, together with a short exact sequence
\[
    0 \longrightarrow M_1 \longrightarrow M \longrightarrow M_2 \longrightarrow
    0.
\]
We then have the commutative diagram
  \[
    \begin{tikzcd}[nodes={execute at begin
      node=\everymath{\displaystyle}},column sep=scriptsize]
      0 \rar
      & \frac{S}{J} \otimes_R F^e_{R*}M_1 \rar\dar[hook]
      & \frac{S}{J} \otimes_R F^e_{R*}M \rar\dar
      & \frac{S}{J} \otimes_R F^e_{R*}M_2 \rar\dar[hook]
      & 0\\
      0 \rar
      & F^e_{S*}\biggl(\frac{S}{J^{[p^e]}} \otimes_R M_1 \biggr) \rar
      & F^e_{S*}\biggl(\frac{S}{J^{[p^e]}} \otimes_R M \biggr) \rar
      & F^e_{S*}\biggl(\frac{S}{J^{[p^e]}} \otimes_R M_2 \biggr) \rar
      & 0
    \end{tikzcd}
  \]
where the rows are exact since $S/J^{[p^t]}$ is flat over $R$ for every $t\ge 0$ by Lemma \ref{lem:ascentkeyiso}$(\ref{lem:ascentkeyisoflatlc})$. The two outer homomorphisms are injective by inductive hypothesis, hence the middle homomorphism is injective by the snake lemma.
\end{proof}

We are now ready to prove Theorem \ref{thm:finjascendscmfi}.

\begin{proof}[Proof of Theorem \ref{thm:finjascendscmfi}] Consider the factorization
  \[
    S \xrightarrow{\id_S \otimes_R F_R} S \otimes_R F_{R*}R
    \xrightarrow{F_{S/R}} F_{S*}S
  \]
of the Frobenius homomorphism $F_S\colon S \to F_{S*}S$, where $F_{S/R}$ is the relative Frobenius homomorphism of Definition \ref{def:radu-andre}. This factorization induces the factorization

  \begin{equation}\label{eq:raduandrefactorlc}
    H^i_\fn(S) \xrightarrow{H^i_\fn(\id_S \otimes_R F_R)} H^i_\fn(S \otimes_R F_{R*}R) \xrightarrow{H^i_\fn(F_{S/R})} H^i_\fn(F_{S*}S)
  \end{equation}
of $H^i_\fn(F_S)$. To show that $S$ is $F$-injective, i.e., that $H^i_\fn(F_S)$ is injective, it suffices to show that the two homomorphisms in \eqref{eq:raduandrefactorlc} are injective.

\par We start by setting up some notation. Since $S/\fm S$ is Cohen--Macaulay, there exists a sequence of elements $y_1,y_2,\ldots,y_n \in S$ whose image in $S/\fm S$ is a system of parameters in $S/\fm S$, hence is also a regular sequence in $S/\fm S$. Now consider the ideal
  \[
    J \coloneqq (y_1,y_2,\ldots,y_n) \subseteq S.
  \]
Note that $\sqrt{{\fm}S + J} = \fn$, and hence $H^i_{\fn}(S) = H^i_{\fm S + J}(S)$ for all $i$.

\par We now show that the first homomorphism in \eqref{eq:raduandrefactorlc} is injective. Applying the functorial isomorphism of Lemma \ref{lem:ascentkeyiso}$(\ref{lem:ascentkeyiso3})$ to the Frobenius homomorphism $F_R\colon R \to F_{R*}R$ and the ideal $I = \fm$ of $R$, we obtain the commutative diagram
  \[
    \begin{tikzcd}[column sep=8.25em,row sep=small]
      H^i_{\fn}(S)
      \rar{H^i_\fn(\id_S \otimes_R F_R)} \arrow[phantom]{d}[sloped]{\simeq} &
      H^i_{\fn}(S \otimes_R F_{R*}R)
      \arrow[phantom]{d}[sloped]{\simeq}\\
      H^n_{J}(S) \otimes_R H^{i-n}_\fm(R)
      \rar[hook]{\id_{H^n_{J}(S)} \otimes_R H^{i-n}_{\fm}(F_R)} &
      H^n_{J}(S) \otimes_R H^{i-n}_{\fm}(F_{R*}R)\mathrlap{.}
    \end{tikzcd}
  \]
The horizontal homomorphism in the bottom row is injective by
the $F$-injectivity of $R$, and then by using the flatness of $H^n_{J}(S)$ over $R$ in Lemma \ref{lem:ascentkeyiso}$(\ref{lem:ascentkeyisoflatlc})$. By the commutativity of the diagram, we see that the first homomorphism in \eqref{eq:raduandrefactorlc} is injective.
  
\par It remains to show that the second homomorphism in \eqref{eq:raduandrefactorlc} is injective.
By Lemma \ref{lem:ascentkeyiso}$(\ref{lem:ascentkeyiso3withfrob})$, this
homomorphism can be identified with the homomorphism
\begin{align}
    H^n_J(S) \otimes_R F_{R*}H^{i-n}_\fm(R) &\longrightarrow
    F_{S*}\bigl(H^n_J(S) \otimes_R H^{i-n}_\fm(R)\bigr)\label{eq:identifiedlcmap}
\intertext{from Lemma \ref{lem:sillylemma}. Since $H^{i-n}_\fm(R)$ is Artinian  \cite[Thm.\ 7.1.3]{BS13}, Proposition \ref{prop:ene219} implies the homomorphisms}
    \frac{S}{J^{[p^e]}} \otimes_R F_{R*}H^{i-n}_\fm(R) &\longrightarrow
    F_{S*}\biggl(\frac{S}{J^{[p^{e+1}]}} \otimes_R H^{i-n}_\fm(R)\biggr)\nonumber
\end{align}
are injective for every integer $e > 0$ (note that $J^{[p^e]}$ is also generated by elements whose images in $S/\fm S$ form a system of parameters). Taking the colimit over all $e$, we see that the homomorphism \eqref{eq:identifiedlcmap} is injective, and hence the second homomorphism in \eqref{eq:raduandrefactorlc} is injective as well.
\end{proof}

\begin{remark}
In some special cases, one can give simpler proofs of Theorem \ref{thm:finjascendscmfi}.
\begin{enumerate}[label=$(\alph*)$]    
    \item If $\varphi$ is regular (resp.\ $F$-pure in the sense of Definition \ref{def:fpuremap} and has Cohen--Macaulay closed fiber), then
      $F_{S/R}$ is faithfully flat by the Radu--Andr\'e theorem (Theorem \ref{thm:raduandre}) (resp.\ pure by definition). The second homomorphism in \eqref{eq:raduandrefactorlc} is therefore injective by the fact that pure ring homomorphisms induce injective maps on Koszul cohomology \cite[Cor.\ 6.6]{HR74}, hence on local cohomology \cite[Exp.\ II, Prop.\ 5]{SGA2}.
      
       \item When $\varphi$ is a regular homomorphism, it is also possible to prove Theorem \ref{thm:finjascendscmfi} using N\'eron--Popescu desingularization \citeleft\citen{Pop86}\citemid Thm.\ 2.5\citepunct\citen{Swa98}\citemid Cor.\ 1.3\citeright, following the strategy in V\'elez's proof of Theorem \ref{thm:velez} in \cite{Vel95}.
\end{enumerate}
\end{remark}

\begin{remark}\label{rem:ascentanalogue-Frational}
The analogue of Theorem \ref{thm:finjascendscmfi} is false for the closely
related notion of $F$-rationality without additional assumptions, since
completions of Gorenstein $F$-rational rings need not be $F$-rational
\cite[\S5]{LR01}, and for completion maps, the closed fiber is even
geometrically regular.
When $R$ and $S$ are excellent, however, $F$-rationality does ascend when the base ring $R$ is $F$-rational and the closed fiber is geometrically $F$-rational \cite[Thm.\ 4.3]{AE03}.
\par The analogue of Theorem \ref{thm:finjascendscmfi} also fails for $F$-purity, even for rings of finite type over a field, by Singh's example of a finite type algebra over a field for which $F$-regularity/$F$-purity does not deform \cite[Prop.\ 4.5]{Sin99}.
\end{remark}

\subsection{Geometrically \emph{F}-injective rings and finitely generated field extensions}\label{subsect:geomfinjarbext}
We show that one can take arbitrary finitely generated field extensions in the definition of geometric $F$-injectivity (Definition \ref{def:FI}). We start with the following field-theoretic result.

\begin{citedlem}[{\cite[\href{https://stacks.math.columbia.edu/tag/04KM}{Tag 04KM}]{stacks-project}}]\label{lem:egaiv2467}
Let $k$ be a field of characteristic $p > 0$, and let $k \subseteq k'$ be a finitely generated field extension. Then, there is a finite purely inseparable extension $k \subseteq k_1$, such that we have a diagram
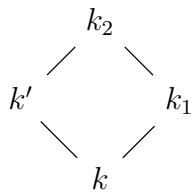
  \[
    \begin{tikzcd}[sep={2.5em,between origins}]
      & k_2 &\\
      k' \arrow[dash]{ur} & & k_1\arrow[dash]{ul}\\
      & k\arrow[dash]{ul}\arrow[dash]{ur} &
    \end{tikzcd}
  \]
of finitely generated field extensions, where $k_1 \subseteq k_2 \coloneqq (k' \otimes_k k_1)_\red$ is a separable field extension.
\end{citedlem}

We can now show that geometric $F$-injectivity is preserved under base change by finitely generated field extensions:

\begin{proposition}\label{prop:geomfinjextensions}
Let $R$ be a Noetherian $k$-algebra, where $k$ is a field of characteristic $p > 0$.
\begin{enumerate}[label=$(\roman*)$,ref=\roman*]
\item Suppose $R$ is $F$-injective, and consider a finitely generated separable field extension $k \subseteq k'$. Then, the ring $R \otimes_k k'$ is $F$-injective.\label{prop:finjseparable}

\item Suppose $R$ is geometrically $F$-injective over $k$, and consider a finitely generated field extension $k \subseteq k'$. Then, the ring $R \otimes_k k'$ is $F$-injective.\label{prop:geomfinjseparable}

\item Suppose $R$ is Cohen--Macaulay and geometrically $F$-injective over $k$, and consider a finitely generated field extension $k \subseteq k'$. Then, the ring $R \otimes_k k'$ is Cohen--Macaulay and $F$-injective.\label{prop:geomCMFIfingen}
\end{enumerate}
\end{proposition}

\begin{proof}
$(\ref{prop:finjseparable})$ follows from Corollary \ref{cor:finjascendscmfinonlocal}, since by separability of $k \subseteq k'$, the ring homomorphism $R \otimes_k k \to R \otimes_k k'$ is regular \cite[Prop.\ 4.6.1]{EGAIV2}.
  
\par We now show $(\ref{prop:geomfinjseparable})$. Let $k_1$ and $k_2$ be as in Lemma \ref{lem:egaiv2467}. Since the homomorphism $R \otimes_k k' \to R \otimes_k k_2$ is faithfully flat by base change, it suffices to show that $R \otimes_k k_2$ is $F$-injective by Theorem \ref{thm:descentF-injective}. Since $R$ is geometrically $F$-injective and $k_1$ is a finite purely inseparable field extension, $R \otimes_k k_1$ is $F$-injective. Moreover, since $k_1 \subseteq k_2$ is a finitely generated separable extension, $(\ref{prop:finjseparable})$ then implies $R \otimes_k k_2 \simeq (R \otimes_k k_1) \otimes_{k_1} k_2$ is $F$-injective.

\par Finally $(\ref{prop:geomCMFIfingen})$ follows from $(\ref{prop:geomfinjseparable})$ and the fact that Cohen--Macaulayness is preserved under base change by finitely generated field extensions \cite[Rem.\ on p.\ 182]{Mat89}. 
\end{proof}

\section{CMFI homomorphisms and openness of \normalfont{\emph{F}}-injective loci}\label{sect:cmfihoms}

After reviewing some basic material on CMFI homomorphisms in \S\ref{subsect:cmfihomsbasic}, we prove Theorem \ref{thm:finjlocusopen}, which says that $F$-injective locus is open in many cases, including in the setting of (\ref{property:openness}). We also give an example of a locally excellent ring for which the $F$-injective locus is not open (Example \ref{ex:finjnotopen}).  In \S\ref{subsect:finjgraded}, we use Theorems \ref{thm:finjascendscmfi} and \ref{thm:finjlocusopen} to show that $F$-injectivity of a graded ring can be detected by localizing at the irrelevant ideal.

\subsection{Definition and properties of CMFI homomorphisms}\label{subsect:cmfihomsbasic}
We begin by defining CMFI homomorphisms and proving some basic properties about
them.

\begin{definition}[cf.\ {\cite[Def.\ 5.4]{Has01}}]\label{def:cmfihom}
Let $\varphi\colon R \to S$ be a flat homomorphism of Noetherian rings.
We say that $\varphi$ is \textsl{$F$-injective} if, for every prime ideal $\fp \subseteq R$, the fiber $S \otimes_R \kappa(\fp)$ of $\varphi$ over $\fp$ is geometrically $F$-injective over $\kappa(\fp)$.
We say that $\varphi$ is \textsl{Cohen--Macaulay $F$-injective} or \textsl{CMFI} if $\varphi$ is Cohen--Macaulay (i.e.\ all fibers of $\varphi$ are Cohen--Macaulay) and $F$-injective. A morphism $f\colon X \to Y$ of locally Noetherian schemes of prime characteristic $p > 0$ is \textsl{CMFI} if $f$ is flat with Cohen--Macaulay and geometrically $F$-injective fibers. 
\end{definition}

As in Remark \ref{rem:gcmiscm}, a flat homomorphism $\varphi$ is CMFI if and only if the fibers of $\varphi$ are geometrically CMFI.

\begin{example}
Since regular homomorphisms are CMFI, if $(R,\fm)$ is an excellent local ring (or more generally, a local $G$-ring), then the canonical map $R \to \widehat{R}$ is CMFI.
\end{example}

We show that the classes of $F$-injective and CMFI homomorphisms
satisfy nice properties.

\begin{lemma}[{cf.\ \cite[Cor.\ 5.6]{Has01}}]\label{lem:finjhomprops}
Let $\varphi\colon R \to S$ and $\psi\colon S \to T$ be homomorphisms of Noetherian rings.
\begin{enumerate}[label=$(\roman*)$,ref=\roman*]
	\item If $\varphi$ is $F$-injective (resp.\ CMFI), then every base change of $\varphi$ along a homomorphism essentially of finite type is $F$-injective (resp.\ CMFI).\label{lem:finjhombasechange}
	
	\item If $\psi$ is faithfully flat and $\psi \circ \varphi$ is $F$-injective (resp.\ CMFI), then $\varphi$ is $F$-injective (resp.\ CMFI).\label{lem:finjhomdescent}

    \item If $\varphi$ is $F$-injective (resp.\ CMFI) and $\psi$ is CMFI, then $\psi \circ \varphi$ is $F$-injective (resp.\ CMFI). \label{lem:finjhomtrans}
    \end{enumerate}
Analogous results hold for $F$-injective and CMFI morphisms of schemes.
\end{lemma}

\begin{proof}
We first show $(\ref{lem:finjhombasechange})$. By \cite[Lem.\ 7.3.7]{EGAIV2}, it suffices to note that if $R$ is a geometrically $F$-injective (resp.\ geometrically CMFI) $k$-algebra, then $R \otimes_k k'$ is also geometrically $F$-injective (resp.\ geometrically CMFI) for every finitely generated field extension $k \subseteq k'$. This property holds by Proposition \ref{prop:geomfinjextensions}$(\ref{prop:geomfinjseparable})$ (resp.\ Proposition \ref{prop:geomfinjextensions}$(\ref{prop:geomCMFIfingen})$).

We next prove $(\ref{lem:finjhomdescent})$. Since $\psi \circ \varphi$ is flat and $\psi$ is faithfully flat, it follows that $\varphi$ is flat.
Consider $\fp \in \Spec(R)$ and consider a finite, pure inseparable field
extension $k \coloneqq \kappa(\fp) \subseteq k'$.
Since $T \otimes_k k'$ is $F$-injective (resp.\ CMFI) by assumption, it follows that $S \otimes_k k'$ is $F$-injective (resp.\ CMFI) by
Theorem \ref{thm:descentF-injective} (resp.\ \cite[Lem.\ 4.6]{Has10}).

For the proof of $(\ref{lem:finjhomtrans})$, consider $\fp \in \Spec(R)$ and
consider a finite, purely inseparable field extension $k \coloneqq \kappa(\fp)
\subseteq k'$.
The induced map
  \[
    \psi \otimes_k \id_{k'} \colon S \otimes_k k' \longrightarrow T \otimes_k
    k'
  \]
is CMFI by $(\ref{lem:finjhombasechange})$ since $S \rightarrow S
\otimes_k k'$ is module-finite. By Corollary
\ref{cor:finjascendscmfinonlocal} (resp.\ \cite[Thm.\ 4.3]{Ene09}),
we find that
$T \otimes_k k'$ is $F$-injective (resp.\ CMFI).
\end{proof}

As a consequence, the class of Noetherian rings with geometrically $F$-injective or geometrically CMFI formal fibers is closed under taking essentially of finite type ring homomorphisms.

\begin{corollary}
Let $R$ be a Noetherian ring of prime characteristic $p > 0$, and let $S$ be an essentially of finite type $R$-algebra. If $R$ has geometrically $F$-injective (resp.\ geometrically CMFI) formal fibers, then so does $S$.
\end{corollary}

\begin{proof}
By \cite[Cor.\ 7.4.5]{EGAIV2}, it suffices to show that the property ``geometrically $F$-injective'' (resp.\ ``geometrically CMFI'') satisfies the assertions (P\textsubscript{I}), (P\textsubscript{II}), and (P\textsubscript{III}) from \cite[(7.3.4)]{EGAIV2}, and the assertion (P\textsubscript{IV}) from \cite[(7.3.6)]{EGAIV2}. First, (P\textsubscript{III}) holds since fields are CMFI. Next, (P\textsubscript{I}) and (P\textsubscript{II}) hold by Lemmas \ref{lem:finjhomprops}$(\ref{lem:finjhomtrans})$ and \ref{lem:finjhomprops}$(\ref{lem:finjhomdescent})$, respectively. Finally, (P\textsubscript{IV}) holds by Proposition \ref{prop:geomfinjextensions}$(\ref{prop:geomfinjseparable})$ (resp.\ Proposition \ref{prop:geomfinjextensions}$(\ref{prop:geomCMFIfingen})$).
\end{proof}

We also find that $F$-injectivity often interacts well with tensor products:

\begin{corollary}
\label{cor:tensor-products-Finj}
Let $k$ be a perfect field of characteristic $p > 0$. If $R$ is a CMFI $k$-algebra, then for any essentially of finite type $F$-injective $k$-algebra $S$, the tensor product $R \otimes_k S$ is $F$-injective.
\end{corollary}

\begin{proof}
Since $k$ is perfect, the $k$-algebra map $k \rightarrow R$ is CMFI. Thus, by essentially of finite type base change (Lemma \ref{lem:finjhomprops}$(\ref{lem:finjhombasechange})$), $S \rightarrow R \otimes_k S$ is also CMFI. Then $R \otimes_k S$ is $F$-injective by ascent of $F$-injectivity under CMFI homomorphisms (Corollary \ref{cor:finjascendscmfinonlocal}) because $S$ is $F$-injective by hypothesis.
\end{proof}

\subsection{Openness of the \emph{F}-injective locus}\label{subsect:finjopen}
We first prove a result on the behavior of $F$-injective loci under CMFI
morphisms.

\begin{proposition}\label{prop:ascdescopen}
Let $f\colon X \rightarrow Y$ be a CMFI morphism of locally Noetherian schemes of prime characteristic $p >0$. Let $\FI(X)$ (resp.\ $\FI(Y)$) be the locus of points in $X$ (resp.\ $Y$) at which $X$ (resp.\ $Y$) is $F$-injective. We then have the following:
\begin{enumerate}[label=$(\roman*)$,ref=\roman*]
  \item $f^{-1}(\FI(Y)) = \FI(X)$.\label{prop:cmfifilocus}
\item If the $F$-injective locus of $Y$ is open, then the $F$-injective locus of
  $X$ is open.\label{prop:cmfifiopen}
\item If $f$ is faithfully flat and quasi-compact, then $\FI(X)$ is open
  if and only if $\FI(Y)$ is open.\label{prop:cmfisurj}
\end{enumerate}
\end{proposition}

\begin{proof}
$(\ref{prop:cmfifiopen})$ follows from $(\ref{prop:cmfifilocus})$ by continuity
of $f$, and $(\ref{prop:cmfisurj})$ follows from $(\ref{prop:cmfifiopen})$ by
\cite[Cor.\ 2.3.12]{EGAIV2} because $f$ is quasi-compact and faithfully flat.
Thus, it suffices to prove $(\ref{prop:cmfifilocus})$.
It suffices to show that for all $x \in X$,
the local ring $\mathcal{O}_{X,x}$ is $F$-injective if and only if $\mathcal{O}_{Y,f(x)}$ is $F$-injective.
This follows from Theorems
\ref{thm:finjascendscmfi} and
\ref{thm:descentF-injective} since
the flat local homomorphism
  \[
    \mathcal{O}_{X,x} \longrightarrow \mathcal{O}_{Y,f(x)}
  \]
  induced by $f$ is CMFI by \cite[Cor.\ 4.11]{Has10}.
\end{proof}

\begin{remark}
Proposition \ref{prop:ascdescopen}$(\ref{prop:cmfifilocus})$ fails for morphisms
with geometrically
CMFI fibers without the flatness hypothesis. Indeed, the canonical map
\[
  \Spec\bigl(k[x]/(x^2)\bigr) \longrightarrow \Spec\bigl(k[x]\bigr)
\]
has geometrically CMFI fibers and $k[x]$ is
$F$-injective, but $k[x]/(x^2)$ is not since it is not reduced (Corollary \ref{cor:finjwn}($\ref{cor:finjred}$)). However, in this
case the $F$-injective locus of $k[x]/(x^2)$ is still open, albeit empty.
\end{remark}
The next result affirmatively answers a question of the second author \cite[Rem.\ 3.6]{Mur}.

\begin{customthm}{B}\label{thm:finjlocusopen}
Let $R$ be a ring essentially of finite type over a Noetherian local ring $(A,\fm)$ of prime characteristic $p > 0$, and suppose that $A$ has Cohen--Macaulay and geometrically $F$-injective formal fibers. Then, the $F$-injective locus is open in $\Spec(R)$.
\end{customthm}

We note that the hypotheses on the formal fibers of $A$ are satisfied when $A$ is excellent, or more generally, a $G$-ring.

\begin{proof}
Let $A \to \widehat{A}$ be the completion of $A$ at $\fm$, and let $\Lambda$ be
a $p$-basis for $\widehat{A}/\fm\widehat{A}$ as in the gamma construction of
Hochster--Huneke (see \cite[$(6.11)$]{HH94} or \cite[Constr.\ 3.1]{Mur}). For every cofinite subset $\Gamma \subseteq \Lambda$, consider the commutative diagram
  \[
    \begin{tikzcd}
      A \dar \rar & \widehat{A} \rar\dar & \widehat{A}^\Gamma\dar\\
      R \rar{\pi} & R \otimes_A \smash{\widehat{A}} \rar{\pi^\Gamma}
      & R \otimes_A \smash{\widehat{A}^\Gamma}
    \end{tikzcd}
  \]
where the squares are co-Cartesian. By the gamma construction \cite[Thm.\ 3.4$(ii)$]{Mur}, there exists a cofinite subset $\Gamma \subseteq \Lambda$ such that $\pi^\Gamma$ induces a homeomorphism on spectra identifying $F$-injective loci. Since $R \otimes_A \widehat{A}^\Gamma$ is $F$-finite, the $F$-injective locus of $\Spec(R \otimes_A \widehat{A}^\Gamma)$ is open by \cite[Lem.\ A.2]{Mur}, and so, the $F$-injective locus of $R \otimes_A \widehat{A}$ is also open. Since an essentially of finite type base change of a faithfully flat CMFI morphism is faithfully flat and CMFI by Lemma \ref{lem:finjhomprops}$(\ref{lem:finjhombasechange})$, it follows that $\pi\colon R \to R \otimes_A \widehat{A}$ is a faithfully flat CMFI homomorphism of Noetherian rings. Thus, the $F$-injective locus of $R$ is open by Proposition \ref{prop:ascdescopen}$(\ref{prop:cmfisurj})$. 
\end{proof}

\begin{remark}\label{rem:gamma3forfinj}
  Let $(A,\fm,k)$ be a Noetherian complete local ring of prime characteristic $p
  > 0$.
  In the proof of \cite[Thm.\ 3.4$(ii)$]{Mur}, it is stated that the following
  property holds when $\cP$ is the property ``$F$-injective'':
  \begin{enumerate}[label=$(\Gamma\arabic*)$,ref=\ensuremath{\Gamma}\arabic*]
      \setcounter{enumi}{2}
    \item\label{axiom:gammaascent}
      For every local ring $B$ essentially of finite type over $A$, if $B$ is
      $\cP$, then there exists a cofinite subset $\Gamma_1 \subseteq \Lambda$
      such that $B^\Gamma \coloneqq B \otimes_A A^\Gamma$ is $\cP$ 
      for every cofinite subset $\Gamma \subseteq
      \Gamma_1$.
  \end{enumerate}
  In \cite{Mur}, the proof of $(\ref{axiom:gammaascent})$ for $F$-injectivity is
  incorrect as stated: the residue field of $B$ is not necessarily a finite
  extension of the residue field $k$ of $A$.
  The proof of the fact that $(\ref{axiom:gammaascent})$ holds for $F$-injectivity 
  is still a straightforward adaptation of the proof of \cite[Lem.\
  2.9$(b)$]{EH08}: it suffices to work with vector spaces over the residue field
  $l$ of $B$ instead of vector spaces over the residue field $k$ of $A$.
\end{remark}

Even though Theorem \ref{thm:finjlocusopen} shows that the $F$-injective locus of a ring which is essentially of finite type over an excellent local ring is open, 
we now use a construction of Hochster \cite{Hoc73} to show that
the $F$-injective locus of an arbitrary locally excellent Noetherian ring of prime characteristic is not necessarily open. Here by a \textsl{locally excellent} Noetherian ring $R$, we mean a ring such that for every prime ideal $\fp \in \Spec(R)$, the localization $R_\fp$ is excellent.

\begin{example}\label{ex:finjnotopen}
Let $(R,\fm)$ be the local ring of a closed
point on an affine variety over an algebraically closed field of characteristic
$p > 0$, such that $R$ is not $F$-injective.
For all $n \in \mathbf{N}$, let $R_n$ be a copy of $R$ with maximal ideal
$\fm_n = \fm$.
Let $R' \coloneqq \bigotimes_{n \in \mathbf N} R_n$, where the infinite tensor
product is taken over $k$, and consider the ring
\[
  T \coloneqq S^{-1}R',
\]
where $S = R' \smallsetminus (\bigcup_n \fm_nR')$.
Then, $T$ is a locally excellent Noetherian ring whose $F$-injective locus is
not open by applying \cite[Prop.\ 2]{Hoc73} when $\mathcal{P} =$ ``$F$-injective''.
  
\par For a specific example of a local ring $(R,\fm)$ of an affine variety that is not $F$-injective, let $k$ be an algebraically closed field of characteristic $2$ and let
\[
  R = k[x^2,xy,y]_{(x^2,xy,y)} \subseteq k[x,y]_{(x,y)}.
\]
Then, $R$ is not weakly normal \cite[Rem.\ 3.7$(iv)$]{Sch09}, hence also not $F$-injective by Corollary \ref{cor:finjwn}$(\ref{cor:finjwnitem})$.

\par Since strong $F$-regularity, $F$-rationality, and $F$-purity all imply
$F$-injectivity (Remark \ref{rem:implications}),
these loci in $\Spec(T)$ are also not open by applying \cite[Prop.\ 2]{Hoc73} to
these properties $\mathcal{P}$.
\end{example}

\subsection{\emph{F}-injectivity of graded rings}\label{subsect:finjgraded}
As a consequence of Theorems \ref{thm:finjascendscmfi} and
\ref{thm:finjlocusopen}, one can show that the irrelevant ideal of a Noetherian graded ring over a field controls the behavior of $F$-injectivity.

Recall that if $R = \bigoplus_{n=0}^\infty R_n$ is a graded ring such that $R_0 = k$ is a field, then $k^* \coloneqq k \smallsetminus \{0\}$ acts on $R$ as follows: any $c \in k^*$ induces a ring automorphism
\[
  \lambda_c\colon R \longrightarrow R,
\] where $\lambda_c$ maps a homogeneous element $t \in R$ of degree $n$ to $c^nt$. Moreover, if $k$ is infinite, then an ideal $I$ of $R$ is homogeneous if and only if $I$ is preserved under this action of $k^*$. In other words, $I$ is homogeneous if for all $c \in k^*$, we have $\lambda_c(I) = I$. With these preliminaries, we have the following result:

\begin{theorem}\label{thm:finjgradedrings}
Let $R = \bigoplus_{n=0}^\infty R_n$ be a Noetherian graded ring such that $R_0
= k$ is a field of characteristic $p > 0$. Let $\fm \coloneqq
\bigoplus_{n=1}^\infty R_n$ be the irrelevant ideal. Then, the following are equivalent:
\begin{enumerate}[label=$(\roman*)$,ref=\roman*]
\item $R$ is $F$-injective.\label{thm:finjgradedringsi}
\item For every homogeneous prime ideal $\fp$ of $R$, the localization $R_\fp$ is $F$-injective.
  \label{thm:finjgradedringsii}
\item $R_\fm$ is $F$-injective.\label{thm:finjgradedringsiii}
\end{enumerate}
\end{theorem}

\begin{proof}
  $(\ref{thm:finjgradedringsi})\Rightarrow(\ref{thm:finjgradedringsii})$
  follows from Proposition \ref{prop:finjlocalizes}, while
  $(\ref{thm:finjgradedringsii})\Rightarrow(\ref{thm:finjgradedringsiii})$ is trivial. To complete the proof of the theorem, it suffices to show $(\ref{thm:finjgradedringsiii})\Rightarrow(\ref{thm:finjgradedringsi})$. 

Let $K \coloneqq k(t)$, where $t$ is an indeterminate, and let $R_K \coloneqq K \otimes_k R$. Since $k \hookrightarrow K$ is regular and faithfully flat, the inclusion
$R \hookrightarrow R_K$
is also regular and faithfully flat. Furthermore, because the irrelevant maximal ideal $\eta$ of $R_K$ is expanded from the irrelevant maximal ideal $\fm$ of $R$, it follows that $R_\fm \hookrightarrow (R_K)_\eta$ is faithfully flat with geometrically regular closed fiber. Thus, by Theorem \ref{thm:finjascendscmfi}, the ring $(R_K)_\eta$ is $F$-injective. If we can show that $R_K$ is $F$-injective, then $R$ will be $F$-injective by Theorem \ref{thm:descentF-injective}.

Theorem \ref{thm:finjlocusopen} shows that the $F$-injective locus of $R_K$ is open because $R_K$ is of finite type over $K$. Let $I$ be the radical ideal defining the complement of the $F$-injective locus. Since the non-$F$-injective locus is preserved under automorphisms of $R$, we see that $I$ is stable under the action of $K^*$.
It follows that $I$ is a homogeneous ideal since $K$ is infinite. Since $\eta$ is in the complement of the closed set defined by $I$, this forces $I$ to equal $R_K$. Thus, $R_K$ is $F$-injective, completing the proof of the Theorem.
\end{proof}

\section{Singularities of generic projection hypersurfaces}\label{sect:doherty}
Bombieri \cite[p.\ 209]{Bom73} and Andreotti--Holm \cite[p.\ 91]{AH77}
asked whether the image of a
smooth projective variety of dimension $r$ under a generic projection to
$\PP^{r+1}_k$ is weakly normal.
As an application of our results on $F$-injectivity, we show that in low
dimensions, we have the following stronger result:

\begin{customthm}{C}\label{thm:dohertyanalogue}
Let $Y \subseteq \PP^n_k$ be a smooth projective variety of dimension $r \le 5$ over an algebraically closed field $k$ of characteristic $p > r$, such that $Y$ is embedded via the $d$-uple embedding with $d \ge 3r$. If $\pi\colon Y \to \PP^{r+1}_k$ is a generic projection and $X = \pi(Y)$, then $X$ is $F$-pure, and hence $F$-injective.
\end{customthm}

\noindent Here when we say $X$ is \textsl{$F$-pure}, we mean that all local
rings of $X$ are $F$-pure (and not that $X$ is globally Frobenius split).
Theorem \ref{thm:dohertyanalogue} is the positive characteristic analogue of a
theorem of Doherty \cite[Main Thm.]{Doh08}, who proved that over the complex
numbers, the image $X$ of the generic projection has semi-log canonical
singularities.\medskip
\par The proof of Theorem \ref{thm:dohertyanalogue} largely follows that of
\cite[Main Thm.]{Doh08}, which checks that the singularities on generic
projections are Du Bois after reduction modulo $p$ using Fedder's criterion
\cite[Prop.\ 2.1]{Fed83}.
Compared to Doherty's proof, the main differences are that we must
be careful about what possibilities occur in Roberts's classification
\cite[(13.2)]{Rob75} for fixed prime characteristics, and we also cannot use
Doherty's characteristic zero proof that the pinch point is Du Bois in
$(\hyperlink{case:roberts1a}{1a})$.
We also need to use Corollary \ref{cor:tensor-products-Finj} instead of
\cite[Thm.\ 3.9]{Doh08}.
\par We will use the following result, the characteristic zero analogue of which was
implicit in the proof of \cite[Main Thm.]{Doh08}.
\begin{lemma}[cf.\ {\cite[Prop.\ 4.8]{Sch09}}]\label{lem:sch09prop48}
  Let $X$ be a reduced locally Noetherian scheme of prime characteristic $p >
  0$.
  Suppose that $X$ can be written as the union of two closed Cohen--Macaulay
  subschemes $Y_1$ and $Y_2$ of the same dimension.
  If $Y_1$, $Y_2$, and $Y_1 \cap Y_2$ are all $F$-injective, then $X$ is
  $F$-injective.
\end{lemma}
Schwede proves Lemma \ref{lem:sch09prop48} under the additional
hypothesis that $X$ is $F$-finite.
While we will only need the $F$-finite case in the sequel, the
proof in \cite{Sch09} applies in the non-$F$-finite case as well, 
once we know that Corollary \ref{cor:globalfinj} holds.

\begin{proof}[Proof of Theorem \ref{thm:dohertyanalogue}]
  By \cite[Thm.\ 1.1]{RZN84}, the generic projection $\pi\colon Y \to X$ is finite
  and birational.
  Thus, $X$ is a hypersurface, and hence also Gorenstein. It therefore suffices to show that $X$ has $F$-injective singularities by \cite[Lem.\ 3.3]{Fed83} (see also Remark \ref{rem:implications}).
  
\par Since $F$-injectivity is unaffected under taking completions \cite[Rem.\ on p.\ 473]{Fed83}, it suffices to show that for every closed point $x \in X$, the completion $\widehat{\cO}_{X,x}$ of the local ring at $x$ is $F$-injective. In \cite[(13.2)]{Rob75}, Roberts provides a list of possibilities for the isomorphism classes of these complete local rings, which we treat individually.
We label each case according to the notation in \cite{Rob75}.
\begin{enumerate}
    \item[$(0)$]\hypertarget{case:roberts0}
      The variety $X$ could analytically locally have simple normal crossings at $x$, i.e.,
      \[
        \widehat{\cO}_{X,x} \simeq \frac{k\llbracket z_1,z_2,\ldots,z_{r+1}
        \rrbracket}{(z_1z_2\cdots z_d)}
      \]
      for some $d \in \{1,2,\ldots,r+1\}$.
      This ring is $F$-pure by Fedder's criterion
      (see \cite[p.\ 2411]{Doh08}).

    \item[$(1a)$]\hypertarget{case:roberts1a}
      When $r \ge 2$ and using the assumption that $p \ne 2$,
      the variety $X$ could have a pinch point at $x$, i.e.,
      \[
        \widehat{\cO}_{X,x} \simeq \frac{k\llbracket z_1,z_2,\ldots,z_{r+1}
        \rrbracket}{(z_r^2-z_1^2z_{r+1})}.
      \]
      This ring is $F$-pure by Fedder's criterion (see \cite[p.\ 160]{BS13p-1}).

   \item[$(1b)$]\hypertarget{case:roberts1b}
      When $r \ge 4$ and using the assumption that $p \ne 3$, we
      could have
      \[
        \widehat{\cO}_{X,x} \simeq \frac{k\llbracket z_1,z_2,\ldots,z_{r+1}
        \rrbracket}{(z_r^3+\Phi_4+\Phi_5)}
      \]
      where
      \begin{align*}
        \Phi_4(z_1,z_2,z_3,z_r,z_{r+1}) &= z_1^2z_3z_r - z_1^3z_{r+1} +
        2z_2z_3z_r^2 - 3z_1z_2z_rz_{r+1},\\
        \Phi_5(z_1,z_2,z_3,z_r,z_{r+1}) &= z_1^2z_3^2z_r - z_1z_2^2z_3z_{r+1} -
        z_2^3z_{r+1}^2.
      \end{align*}
      This ring is $F$-pure by Fedder's criterion (see \cite[p.\ 2412]{Doh08}).
  \end{enumerate}
  The remaining cases are analytically locally unions of two hypersurfaces.
  By Lemma \ref{lem:sch09prop48}, it suffices to show that the hypersurfaces and
  their intersections are $F$-injective.
  \begin{enumerate}
   \item[$(2a)$]\hypertarget{case:roberts2a}
      When $r \ge 3$ and using the assumption that $p \ne 2$, the variety $X$
      could analytically locally be the union of a hyperplane $Y_1$ and a pinch
      point $Y_2$ at $x$, i.e.,
      \[
        \widehat{\cO}_{X,x} \simeq \frac{k\llbracket z_1,z_2,\ldots,z_{r+1}
        \rrbracket}{z_1(z_r^2-z_2^2z_{r+1})}.
      \]
      The hyperplane $Y_1$ is regular, hence $F$-injective, and both $Y_2$ and the
      intersection $Y_1 \cap Y_2$ are pinch points, hence are $F$-injective by
      $(\hyperlink{case:roberts1a}{1a})$.

    \item[$(2b)$] When $r = 5$ and using the assumption that $p \ne 3$, the
      variety $X$ could analytically locally be the union of a hyperplane $Y_1$
      and a hypersurface $Y_2$ of the form in $(\hyperlink{case:roberts1b}{1b})$,
      i.e.,
      \[
        \widehat{\cO}_{X,x} \simeq \frac{k\llbracket z_1,z_2,\ldots,z_{r+1}
        \rrbracket}{z_1(z_r^3+\Psi_4+\Psi_5)}
      \]
      where with notation as in $(\hyperlink{case:roberts1b}{1b})$,
      we set $\Psi_i = \Phi_i(z_2,z_3,z_4,z_r,z_{r+1})$ for $i \in
      \{4,5\}$.
      The hyperplane $Y_1$ is regular, hence $F$-injective, and both $Y_2$ and
      $Y_1 \cap Y_2$ are of the form in $(\hyperlink{case:roberts1b}{1b})$,
      hence are $F$-injective.

    \item[$(2c)$]\hypertarget{case:roberts2c}
      When $r = 5$ and using the assumption that $p \ne 2$, the
      variety $X$ could analytically locally be the union of two pinch points
      $Y_1$ and $Y_2$ at $x$, i.e.,
      \[
        \widehat{\cO}_{X,x} \simeq \frac{k\llbracket z_1,z_2,\ldots,z_{r+1}
        \rrbracket}{(z_r^2-z_1^2z_{r+1})(z_{r-2}^2-z_2^2z_{r-1})}.
      \]
      Both $Y_1$ and $Y_2$ are $F$-injective by
      $(\hyperlink{case:roberts1a}{1a})$.
      To show that $Y_1 \cap Y_2$ is $F$-injective, it suffices to show that the local
      ring
      \[
        \left(
          \frac{k[z_1,z_r,z_{r+1}]_{(z_1,z_r,z_{r+1})}}{(z_r^2-z_1^2z_{r+1})} \otimes_k
        \frac{k[z_2,z_3,\ldots,z_{r-2},z_{r-1}]_{(z_2,z_3,\ldots,z_{r-2},z_{r-1})}}{(z_{r-2}^2-z_2^2z_{r-1})}
      \right)_{(z_1,z_2,\ldots,z_{r+1})}
      \]
      is $F$-injective, since $F$-injectivity is unaffected by completion
      \cite[Rem.\ on p.\ 473]{Fed83}.
      Each factor of the tensor product is CMFI by the fact 
      that their completions are $F$-injective by using $(\hyperlink{case:roberts1a}{1a})$. 
      Thus, the tensor product is $F$-injective by Corollary \ref{cor:tensor-products-Finj}. Finally, 
      the further localization of the tensor product is $F$-injective since $F$-injectivity
      localizes by Proposition \ref{prop:finjlocalizes}. 
    
    \item[$(3)$]\hypertarget{case:roberts3} When $r \ge 4$ and using the
      assumption that $p \ne 2$, the variety $X$ could analytically locally be
      the union of a simple normal crossing divisor $Y_1$ and a pinch point
      $Y_2$ at $x$, i.e.,
      \[
        \widehat{\cO}_{X,x} \simeq \frac{k\llbracket z_1,z_2,\ldots,z_{r+1}
        \rrbracket}{z_1z_2(z_r^2-z_3^2z_{r+1})}.
      \]
      The hypersufaces $Y_1$ and $Y_2$ are $F$-injective by $(\hyperlink{case:roberts0}{0})$
      and $(\hyperlink{case:roberts1a}{1a})$, respectively.
      As in $(\hyperlink{case:roberts2c}{2c})$, to show that $Y_1 \cap Y_2$ is
      $F$-injective, it suffices to show that the local ring
      \[
        \left(\frac{k[z_1,z_2]_{(z_1,z_2)}}{(z_1z_2)} \otimes_k
        \frac{k[z_3,z_4,\ldots,z_{r+1}]_{(z_3,z_4,\ldots,z_{r+1})}}{(z_r^2-z_3^2z_{r+1})}
        \right)_{(z_1,z_2,z_3,z_4,\ldots,z_{r+1})}
      \]
      is $F$-injective.
      The inner tensor product ring is $F$-injective by Corollary \ref{cor:tensor-products-Finj}
      since each factor is CMFI by the fact that the completions of these
      factors are CMFI via
      $(\hyperlink{case:roberts0}{0})$ and $(\hyperlink{case:roberts1a}{1a})$,
      respectively. Finally, as in $(\hyperlink{case:roberts2c}{2c})$, a 
      further localization does not affect $F$-injectivity.

    \item[$(4)$] When $r = 5$ and using the assumption that $p \ne 2$, the
      variety $X$ could analytically locally be the union of a simple normal
      crossing divisor $Y_1$ and a pinch point $Y_2$ at $x$, i.e.,
      \[
        \widehat{\cO}_{X,x} \simeq \frac{k\llbracket z_1,z_2,\ldots,z_{r+1}
        \rrbracket}{z_1z_2z_3(z_r^2-z_4^2z_{r+1})}.
      \]
      For this, the same reasoning as in
      $(\hyperlink{case:roberts3}{3})$ applies.
      \qedhere
  \end{enumerate}
\end{proof}

On the other hand, we can adapt an example of Doherty \cite[Cor.\ 4.7 and Ex.\ 4.8]{Doh08} to show that generic projections of large dimension cannot be $F$-pure.

\begin{proposition}[{cf.\ \cite[Cor.\ 4.7]{Doh08}}]\label{prop:doh47}
Let $Y \subseteq \PP^n_k$ be a smooth projective variety of dimension $30$ over an algebraically closed field of characteristic $p > 0$ such that $\Omega_Y^1$ is nef. If $\pi\colon Y \to \PP^{31}_k$ is a generic projection and $X = \pi(Y)$, then $X$ is not $F$-pure.
\end{proposition}

\begin{proof}
  The proof of \cite[Cor.\ 4.7]{Doh08} shows that in this setting, there are
  closed points $x \in X$ with Hilbert--Samuel multiplicity at least $2^5 = 32$.
  Now let $f \in k[x_1,x_2,\ldots,x_{31}]$ be a local defining equation of $X$
  at $x$, and let $\fm_x \subseteq k[x_1,x_2,\ldots,x_{31}]$ be the maximal
  ideal defining $x \in X$.
  By the multiplicity condition, we have $f \in \fm_x^{32}$.
  Thus, we have
  \[
    f^{p-1} \in \fm_x^{32(p-1)} \subseteq \fm_x^{31(p-1)+1} \subseteq
    \fm_x^{[p]},
  \]
  where the second inclusion
  follows from the pigeonhole principle \cite[Lem.\ 2.4$(a)$]{HH02}.
  Fedder's criterion \cite[Prop.\ 2.1]{Fed83} therefore implies that $X$ is not
  $F$-pure at $x$.
\end{proof}
\begin{remark}
  The calculation in \cite[Thm.\ 4.5]{Doh08} implies that the hypersurface $X$
  in Proposition \ref{prop:doh47} is also not semi-log canonical.
\end{remark}
Abelian varieties give examples of smooth projective varieties satisfying the
hypothesis of Proposition \ref{prop:doh47}, since they have nef cotangent
bundles.
We give some other examples of such varieties.
\begin{example}\label{ex:dohertyhighdim}
  Over the complex numbers, there are many explicit examples of smooth
  projective varieties $Y$ with ample cotangent bundle $\Omega^1_Y$ (see
  \cite[Constrs.\ 6.3.36--6.3.39]{Laz04b}).
  In positive characteristic, we know of two methods to produce such examples.
  \begin{enumerate}[label=$(\alph*)$]
    \item Fix an integer $r > 0$ and an integer $n \ge 3r-1$.
      Following the proof of \cite[Constr.\ 6.3.42]{Laz04b}, general
      complete intersections of $(n-r)$ hypersurfaces of sufficiently large
      degree in a product of $n$ curves of genus $\ge 2$ are $r$-dimensional
      smooth projective varieties with ample cotangent bundles.
    \item Fix an integer $r > 0$ and an integer $n \ge 2r$.
      By Xie's proof of Debarre's ampleness conjecture \cite[Thm.\
      1.2]{Xie18}, general complete intersections $X \subseteq \PP^n_k$ of
      $(n-r)$ hypersurfaces of sufficiently large degree are $r$-dimensional
      smooth projective varieties with ample contangent bundles.
  \end{enumerate}
  By Proposition \ref{prop:doh47}, applying either construction when $r = 30$
  yields examples of smooth projective varieties whose generic projections are
  not $F$-pure.
\end{example}

\appendix
\section{\normalfont{\emph{F}}-rationality descends and implies \normalfont{\emph{F}}-injectivity}
We address the relationship between $F$-rationality and $F$-injectivity and the descent property (\ref{property:descent}) for $F$-rationality.
We start by defining $F$-rationality.
\begin{citeddef}[{\citeleft\citen{FW89}\citemid Def.\ 1.10\citepunct
  \citen{HH90}\citemid Def.\ 2.1\citepunct
  \citen{HH94}\citemid Def.\ 4.1\citeright}]\label{def:frational}
  Let $R$ be a Noetherian ring. A sequence of elements $x_1,x_2,\ldots,x_n \in R$ is a \textsl{sequence of parameters} if, for every prime ideal $\fp \subseteq R$ containing $(x_1,x_2,\ldots,x_n)$, the images of $x_1,x_2,\ldots,x_n$ in $R_\fp$ are part of a system of parameters in $R_\fp$. An ideal $I \subseteq R$ is a \textsl{parameter ideal} if $I$ can be generated by a sequence of parameters in $R$.
  \par Now let $R$ be a Noetherian ring of prime characteristic $p > 0$. We say that $R$ is \textsl{$F$-rational} if every parameter ideal in $R$ is tightly closed in the sense of \cite[Def.\ 3.1]{HH90}.
\end{citeddef}

We now prove the following lemma, which allows us to spread out a system of parameters in a localization to a sequence of parameters in the whole ring.

\begin{lemma}[cf.\ {\cite[Proof of Prop.\ 6.9]{QS17}}]\label{lem:liftsop}
  Let $R$ be a Noetherian ring, and consider a prime ideal $\fp \subseteq R$
  of height $t$.
  Let $I \subseteq R_\fp$ be an ideal generated by a system of parameters.
  Then, there exists a sequence of parameters $x_1,x_2,\ldots,x_t \in \fp$ such
  that $I = (x_1,x_2,\ldots,x_t)R_\fp$.
\end{lemma}
\begin{proof}
  Note that when $\height \fp = t = 0$, we do not have anything to prove. So assume
  $t \geq 1$.
  Write $I = (a_1,a_2,\ldots,a_t)R_\fp$, where $t$ is the height of $\fp$ and
  $a_i \in R$ for every $i$.
  Note that $\fp$ is minimal over $(a_1,a_2,\ldots,a_t)$.
  Let $J$ be the $\fp$-primary component of $(a_1,a_2,\ldots,a_t)$ in $R$.
  Then, we have $I = JR_\fp$ and $\height J = t$.
  \par We now claim that there exist elements $b_1,b_2,\ldots,b_t \in J^2$ such
  that setting $x_i = a_i + b_i$, the sequence $x_1,x_2,\ldots,x_t$ is a
  sequence of parameters in $R$. Note that for each $i$, the ideal $(a_i) + J^2$ has
  height $t$ because it is sandwiched between the two ideals $J^2$ and $J$, both
  of which have
  height t.
  For $i = 1$, we have
  \begin{align*}
    (a_1) + J^2 \not\subseteq{}& \bigcup_{\fq \in \Min_R(R)} \fq
    \intertext{by prime avoidance.
    Thus, by a theorem of Davis \cite[Thm.\ 124]{Kap74}, there exists
    $b_1 \in J^2$ such that}
    x_1 \coloneqq a_1 + b_1 \notin{}& \bigcup_{\fq \in \Min_R(R)} \fq.
  \end{align*}
  For every $1 < i \le t$, the same method implies there exist $b_i \in J^2$
  such that
  \[
    x_i \coloneqq a_i + b_i \notin \bigcup_{\Min_R(R/(x_1,x_2,\ldots,x_{i-1}))}
    \fq.
  \]
  By construction, $x_1,x_2,\ldots,x_t$ generates an ideal of height $t$ in $R$,
  and therefore it is a sequence of parameters by \cite[\S2.4, Prop.\ 13]{Bos13},
  since the images of $x_1,x_2,\ldots,x_t$ continue to generate an ideal of height $t$ 
  in $R_\fp$ for any prime $\fp$ that contains these elements.
  Since $(x_1,x_2,\ldots,x_t)R_\fp \subseteq I$ and $I =
  (x_1,x_2,\ldots,x_t)R_\fp + I^2$, Nakayama's lemma implies $I =
  (x_1,x_2,\ldots,x_t)R_\fp$ (see \cite[Cor.\ to Thm.\ 2.2]{Mat89}).
\end{proof}

The following result shows that $F$-rational rings are $F$-injective and are even CMFI under mild assumptions.

\begin{proposition}
\label{prop:frational}
Let $R$ be an $F$-rational Noetherian ring of prime characteristic $p > 0$. We then have the following:
\begin{enumerate}[label=$(\roman*)$,ref=\roman*] 
    \item If $R$ is the homomorphic image of a Cohen--Macaulay ring or is
      locally excellent, then $R_\fm$ is $F$-rational for every
      maximal ideal $\fm \subseteq R$, and $R$ is Cohen--Macaulay.\label{prop:frationalcm}
    
        \item $R$ is $F$-injective.\label{prop:frationalfinj}
\end{enumerate}
\end{proposition}
\begin{proof}
  $(\ref{prop:frationalcm})$ is shown in \citeleft\citen{HH94}\citemid Thm.\
  4.2\citepunct \citen{Vel95}\citemid Prop.\ 0.10\citeright.
  We now show $(\ref{prop:frationalfinj})$.
  By Lemma \ref{lem:cmfiparam}, it suffices to show that every ideal $I
  \subseteq R_\fm$ generated by a system of parameters in $R_\fm$ is Frobenius
  closed.
  By Lemma \ref{lem:liftsop}, there exists a sequence of parameters
  $x_1,x_2,\ldots,x_t \in \fm$ such that $I = (x_1,x_2,\ldots,x_t)R_\fm$.
  Since $R$ is $F$-rational, we see that $(x_1,x_2,\ldots,x_t)$ is tightly
  closed in $R$, hence also Frobenius closed in $R$.
  Since Frobenius closure commutes with localization \cite[Lem.\ 3.3]{QS17}, we
  see that $I$ is Frobenius closed in $R_\fm$.
\end{proof}

\begin{remark}\label{rem:implications}
  For Noetherian rings of prime characteristic $p > 0$, the relationship between
  $F$-injectivity and other classes of singularities can be summarized as
  follows:
  
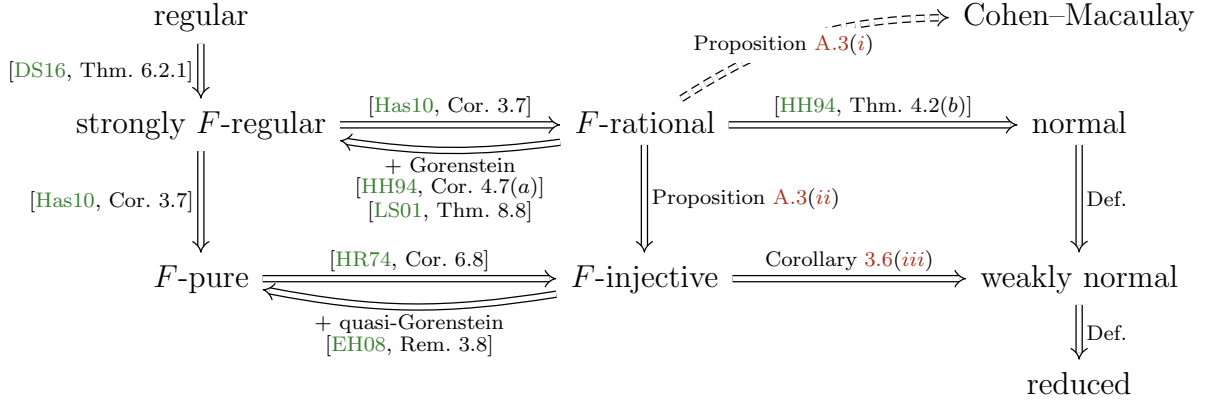
\[
    \begin{tikzcd}[column sep=7em]
      \text{regular} \dar[Rightarrow,swap]{\text{\cite[Thm.\ 6.2.1]{DS16}}}
      & & \text{Cohen--Macaulay}\\
      \text{strongly $F$-regular} \rar[Rightarrow]{\text{\cite[Cor.\
      3.7]{Has10}}} \dar[Rightarrow,swap]{\text{\cite[Cor.\ 3.7]{Has10}}}
      & \text{$F$-rational} \dar[Rightarrow]{\text{Proposition
      \ref{prop:frational}$(\ref{prop:frationalfinj})$}}
      \rar[Rightarrow]{\text{\cite[Thm.\
      4.2$(b)$]{HH94}}} \lar[Rightarrow,bend left=8,start
      anchor={[yshift=1pt]south west},end anchor={[yshift=4pt]south
      east}]{\substack{\text{$+$ Gorenstein}\\\text{\cite[Cor.\
      4.7$(a)$]{HH94}}\\\text{\cite[Thm.\ 8.8]{LS01}}}}
      \arrow[Rightarrow,dashed,bend left=16,end
      anchor=west]{ur}[description,pos=0.4]{%
      \text{\citeleft\citen{HH94}\citemid Thm.\
      4.2\citepunct \citen{Vel95}\citemid Prop.\ 0.10\citeright}}
      & \text{normal}\dar[Rightarrow]{\text{Def.}}\\[1.5em]
      \text{$F$-pure} \rar[Rightarrow]{\text{\cite[Cor.\ 6.8]{HR74}}}
      & \text{$F$-injective} \rar[Rightarrow]{\text{Corollary
      \ref{cor:finjwn}$(\ref{cor:finjwnitem})$}}
      \lar[Rightarrow,bend left=10]{\substack{\text{$+$
      quasi-Gorenstein}\\\text{\cite[Rem.\ 3.8]{EH08}}}}
      & \text{weakly normal}\dar[Rightarrow]{\text{Def.}}\\
      & & \text{reduced}\mathrlap{.}
    \end{tikzcd}
  \]
  Here, strong $F$-regularity is defined in terms of tight closure
  \cite[Def.\ 3.3]{Has10}, following Hochster.
  The dashed implication holds if either $R$ is the homomorphic image of a
  Cohen--Macaulay ring, or is locally excellent.
  We note that the
  implications with citations to the present paper may be well-known to experts.
\end{remark}

Finally, for the sake of completeness, we show that $F$-rationality descends for arbitrary Noetherian rings. This is shown in \cite{Vel95} assuming the existence of test elements.

\begin{proposition}[cf.\ {\cite[(6) on p.\ 440]{Vel95}}]\label{prop:frationaldescends}
  Let $\varphi\colon R \to S$ be a faithfully flat homomorphism of Noetherian
  rings of prime characteristic $p > 0$.
  If $S$ is $F$-rational, then $R$ is $F$-rational.
\end{proposition}
\begin{proof}
Let $I$ be a parameter ideal in $R$. We first claim that $\varphi(I)S$ is a parameter ideal in $S$. Let $\fq \subseteq S$ be a prime ideal containing $\varphi(I)S$, and let $\fp \subseteq R$ be the contraction of $\fq$ in $R$. Since $I$ is a parameter ideal in $R$, we can write $I =
(x_1,x_2,\ldots,x_t)$, where the images of $x_1,x_2,\ldots,x_t$ in $R_\fp$ can be completed to a system of parameters $x_1,x_2,\ldots,x_t,x_{t+1},\ldots,x_d$ for $R_\fp$. Since the induced map $R_\fp \rightarrow S_\fq$ is also faithfully flat, the images of $x_1,x_2,\ldots,x_t,x_{t+1},\ldots,x_d$ in $S_\fq$ form part of a system of parameters of $S_\fq$. We therefore see that $\varphi(I)S$ is indeed a parameter ideal in $S$.

\par Returning to the proof of the Proposition, let $I$ be a parameter ideal in $R$. Since $\varphi(R^\circ) \subseteq S^\circ$ by the faithful flatness of $\varphi$, we have that
  \[
    \varphi(I)S \subseteq \varphi(I^*_R)S \subseteq \bigl(\varphi(I)S\bigr)^*_S
    = \varphi(I)S,
  \]
where the last equality holds by the $F$-rationality of $S$ and the fact that $\varphi(I)S$ is a parameter ideal. Since the leftmost and rightmost ideals are equal, we therefore have $\varphi(I)S = \varphi(I^*_R)S$. Contracting the expansions back to $R$, we then get $I = I^*_R$.
\end{proof}

\begin{remark}
Let $\cP$ be the property of Noetherian local rings being $F$-rational. In Proposition \ref{prop:frationaldescends} we showed that $\cP$ satisfies (\ref{property:descent}), and V\'elez showed that $\cP$ satisfies (\ref{property:openness}) \cite[Thm.\ 3.5]{Vel95}. Moreover, $\cP$ also satisfies (\ref{property:localization}) when $R$ is the homomorphic image of a Cohen--Macaulay ring (Proposition \ref{prop:frational}$(\ref{prop:frationalcm})$), and (\ref{property:ascentstar}) holds when $R$ and $S$ are locally excellent \cite[Thm.\ 3.1]{Vel95}. However, to the best of our knowledge, it is not known if the localization of an $F$-rational local ring is always $F$-rational for arbitrary Noetherian rings, or if an arbitrary 
$F$-rational Noetherian ring is Cohen--Macaulay.
\end{remark}

\end{document}